\documentclass[a4paper,10pt]{amsart}
\usepackage[a4paper]{geometry}
\usepackage[utf8]{inputenc}
\usepackage{amsthm,amsmath,amsfonts,amssymb}
\usepackage{bm,bbm}
\usepackage{mathtools}
\usepackage{appendix}
\usepackage{mathrsfs}
\usepackage{enumitem}
\usepackage{setspace}
\usepackage{todonotes}
\usepackage{thmtools} 
\usepackage[foot]{amsaddr}
\usepackage[pdfdisplaydoctitle,colorlinks,breaklinks,urlcolor=blue,linkcolor=blue,citecolor=blue]{hyperref} 
\usepackage[nameinlink,capitalise]{cleveref}

\newcommand{\C}{\mathbb{C}}

\newcommand{\EE}{\mathbb{E}}

\newcommand{\cF}{\mathcal{F}}

\newcommand{\cH}{\mathcal{H}}

\newcommand{\cQ}{\mathcal{Q}}

\newcommand{\N}{\mathbb{N}}

\newcommand{\PP}{\mathbb{P}}

\newcommand{\R}{\mathbb{R}}
\newcommand{\RR}{\mathcal{R}}

\newcommand{\T}{\mathbb{T}}

\DeclareMathOperator{\re}{Re}

\DeclareMathOperator{\res}{Res}

\renewcommand{\epsilon}{\varepsilon}
\renewcommand{\setminus}{\smallsetminus}

\newcommand{\one}{\bm{1}}

\newcommand{\set}[1]{\left\{#1\right\}}
\newcommand{\pa}[1]{\left(#1\right)}
\newcommand{\bra}[1]{\left[#1\right]}
\newcommand{\abs}[1]{\left|#1\right|}

\newcommand{\brak}[1]{\left\langle#1\right\rangle}

\newcommand{\gamhalf}[1]{\Gamma\left(\frac{#1}{2}\right)}

\newtheorem{theorem}{Theorem}[section]
\newtheorem{definition}[theorem]{Definition}

\newtheorem{lemma}[theorem]{Lemma}
\newtheorem{proposition}[theorem]{Proposition}

\theoremstyle{remark}
\newtheorem{remark}[theorem]{Remark}

\numberwithin{equation}{section}

\newcommand{\dd}{\,\mathrm{d}}
\newcommand{\eps}{\varepsilon}

\newenvironment{acknowledgements}{%
  \begin{abstract}
}{%
  \end{abstract}
}

\title[Anomalous Regularization in Kraichnan's Model]{Anomalous Regularization in Kraichnan's Passive Scalar Model}
\author[L. Galeati]{Lucio Galeati}
\address{Università degli studi dell'Aquila, Dipartimento di Ingegneria e Scienze dell'Informazione e Matematica, Via Vetoio, 67100 L'Aquila, Italy.}
\email{lucio.galeati at univaq.it}
\author[F. Grotto]{Francesco Grotto}
\address{Università di Pisa, Dipartimento di Matematica, Largo Bruno Pontecorvo 5, 56127 Pisa, Italy.}
\email{francesco.grotto at unipi.it}
\author[M. Maurelli]{Mario Maurelli}
\address{Università di Pisa, Dipartimento di Matematica, Largo Bruno Pontecorvo 5, 56127 Pisa, Italy.}
\email{mario.maurelli at unipi.it}
\keywords{passive scalar transport, scalar turbulence, Kraichnan's model, anomalous regularization}
\date\today

\begin{document}

\begin{abstract}
We consider the advection of a passive scalar by a divergence free random Gaussian field, white in time and H\"older regular in space (rough Kraichnan's model), a well-established synthetic model of passive scalar turbulence.
By studying the evolution of negative Sobolev norms, we show an anomalous regularization effect induced by the dynamics: distributional initial conditions immediately become functions of positive Sobolev regularity.\\[1ex]
%
\noindent \textbf{MSC (2020):} 76F55, 76M35, 76F25, 35R36.

\end{abstract}

\maketitle


\section{Introduction}

The study of passive scalar fields $\rho^\nu$ advected by an incompressible turbulent field $u$,
\begin{equation}\label{eq:intro_passive_scalar}
    \partial_t \rho^\nu(t,x)+ u(t,x)\cdot \nabla \rho^\nu(t,x)=\nu \Delta \rho^\nu(t,x),
\end{equation}
is of fundamental importance in Physics and engineering applications and dates back to the classical works of Obukhov \cite{Obukhov1949}, Yaglom \cite{Yaglom1949}, Corrsin \cite{Corrsin1951}, Batchelor \cite{Batchelor1959I,Batchelor1959II}.
Although simpler in nature, the dynamics of passive scalars $\rho^\nu$ shares many phenomenological parallels with the nonlinear turbulent flow of the velocity field $u$, to the point where one can legitimately talk of \emph{scalar turbulence} \cite{shraiman2000scalar,sreenivasan2019turbulent}.
In particular, $\rho^\nu$ is expected to exhibit turbulent mixing, intermittency and most importantly \emph{anomalous dissipation} of energy, in the sense that
\begin{equation}\label{eq:intro_anomalous_dissipation}
    \liminf_{\nu\to 0^+} \int_0^\tau \nu \| \nabla\rho^\nu_t\|_{L^2}^2 \dd t > 0 \quad \forall\, \tau>0.
\end{equation}
Condition \eqref{eq:intro_anomalous_dissipation} necessarily implies, upon passing to the limit in weak topologies, the existence of solutions $\rho$ to the inviscid PDE
\begin{equation}\label{eq:intro_inviscid_PDE}
    \partial_t \rho(t,x)+ u(t,x)\cdot \nabla \rho(t,x)=0
\end{equation}
for which energy is not preserved, i.e. $\| \rho_\tau\|_{L^2} < \| \rho_0\|_{L^2}$ for $\tau>0$, contrary to what a naive integration by parts would suggest.
In particular, the validity of \eqref{eq:intro_anomalous_dissipation} requires some roughness of both $u$ and $\rho$: otherwise, the solution to \eqref{eq:intro_inviscid_PDE} could be represented as
\begin{equation}\label{eq:intro_Lagrangian_formula}
    \rho(t,x) = \rho_0(\Phi^{-1}(t,x))
\end{equation}
where $\Phi^{-1}(t,x)$ is the backward flow associated to the ODE $\dot x_t = u(t,x_t)$. In this perspective, anomalous dissipation requires failure of renormalizability of $\rho$ in the sense of DiPerna--Lions \cite{DiPLio1989}, lack of $W^{1,p}$-regularity for $u$ and possibly absence of a canonical backward solution flow associated to $\dot x_t = u(t,x_t)$.

Due to all these pathological behaviors, verifying \eqref{eq:intro_anomalous_dissipation} in cases of interest remains an extremely hard task. 
There has been however major progress in the mathematical community in recent years, with several explicit constructions of velocity fields $u$ proposed.
We mention the works \cite{DEIJ2022,CoCrSo2023,elgindi2023norm} for examples based on alternating shear flows, the recent breakthrough \cite{armstrong2023anomalous} based on fractal homogenization and the subsequent improvement \cite{burczak2023anomalous} using convex integration.

On the other hand, if one allows for randomly sampled velocity fields $u$, there is a relatively simple, yet surprisingly rich, class of models showing anomalous dissipation.
In his seminal work \cite{Kraichnan1968} in 1968,
Kraichnan considered a random Gaussian velocity field exhibiting rapid fluctuations in time, and evaluated 3D energy spectra by means of Lagrangian-History Direct Interaction (LHDI) closure approximation.
Subsequent works \cite{kraichnan1994,BeGaKu1998,MajKra1999,FaGaVe2001} popularized the advection by noise of transport type, delta-correlated in time and colored in space, as \emph{Kraichnan's model} of passive scalar turbulence \cite{shraiman2000scalar}.

We consider advection of scalars in $\R^d$, $d\geq 2$, by a Gaussian velocity field $\dot W=\dot W(t,x)$ which is white in time and divergence-free, homogeneous and isotropic in space; that is, we take $W$ to be the centered $\R^d$-valued Gaussian process with covariance
\begin{equation}\label{eq:covariance_W}
    \EE[W(t,x)\otimes W(s,y)]=(t\wedge s) Q(x-y),
\end{equation}
with the covariance tensor $Q$ having a power-law spectrum: for some parameter $\alpha\in (0,+\infty)$, we set
\begin{equation}\label{eq:kraichnan_noise}
    Q(z)_{jk}= \frac{1}{(2\pi)^{d/2}}\int_{\R^d} \pa{\delta_{jk}-\frac{\xi_j \xi_k}{|\xi|^2}}  \frac{e^{i\xi\cdot z}}{(1+|\xi|^2)^{d/2+\alpha}}  \dd \xi.
\end{equation}
In the inviscid case, the evolution of the scalar field $\rho$ is thus formally given by the stochastic PDE (compare with \eqref{eq:intro_inviscid_PDE})
\begin{equation}\label{eq:intro_kraichnan}
    \dd \rho(t,x) + \circ \dd W(t,x)\cdot \nabla \rho(t,x) = 0,
\end{equation}
interpreted - at least for regular solutions $\rho$ - in its integral version with Stratonovich stochastic integration in time.
The value $\alpha$ in \eqref{eq:kraichnan_noise} determines sharply the space regularity of $W$, which can be shown to be almost $C^\alpha$ and nowhere better.

As nicely illustrated in the lecture notes \cite{CaFaGa2008}, the Kraichnan model exhibits a phase transition depending on the value of $\alpha$:
\begin{itemize}
    \item[i)] For $\alpha>1$ (\emph{smooth Kraichnan}), $W$ is spatially Lipschitz and admits a unique underlying stochastic flow $\Phi(t,x,\omega)$.
    Solutions to \eqref{eq:intro_kraichnan} can be represented by $\rho(t,x)=\rho_0(\Phi^{-1}(t,x,\omega))$, are renormalized and conserve all $L^p$-norms.
    Lyapunov exponents are deterministic and can be explicitly computed, see \cite{BaxHar1986,LeJan1985}, and the top one is strictly positive, implying chaotic behavior.
    On compact manifolds, the stochastic flow is an exponential mixer \cite{DoKaKo2004,GesYar2021}.
    \item [ii)] For $\alpha\in (0,1)$ (\emph{rough Kraichnan}), solutions to the SPDE \eqref{eq:intro_kraichnan} exhibit anomalous dissipation and intermittency \cite{BeGaKu1998}. The underlying SDE now allows for Lagrangian particle splitting (\emph{spontaneous stochasticity}) and the top Lyapunov exponent is formally $+\infty$ (\emph{hyperinstability}). The dynamics is not described by a flow of maps anymore, rather a flow of Markovian kernels \cite{LeJRai2002}, which retains a diffusive behavior, in the sense that, formally at least, initial conditions $\rho_0=\delta_x$ become $L^1$-valued at positive times.
\end{itemize}

Despite this counterintuitive behavior, it is worth noting that in the regime $\alpha\in (0,1)$ existence and uniqueness of solutions to \eqref{eq:intro_kraichnan} still hold, which are moreover recovered in the vanishing viscosity limit.
The other main known example of a PDE displaying all these features (uniqueness of vanishing viscosity limit, anomalous dissipation, intermittency) is the 1D Burgers equation, see \cite{DriEyi2015} and references therein.

The rough Kraichnan model has been the source of fundamental new concepts in turbulence theory, such as Lagrangian slow modes and spontaneous stochasticity; the latter has then been understood to be deeply related to anomalous dissipation \eqref{eq:intro_anomalous_dissipation} in a much more general setting, see \cite{DriEyi2017}. For $\alpha=2/3$, it reproduces the Richardson-Kolmogorov scaling of energy cascade (cf. \cite[pp. 23,84]{CaFaGa2008}), as well as predictions from the Obukhov-Corrsin theory of passive scalar turbulence.
Despite its simplicity, it remains a great tool to develop and test new ideas in isotropic turbulence.\\

In this work, by studying the evolution of multiscale norms of solutions $\rho(t,x)$ to \eqref{eq:intro_kraichnan}, we observe the remarkable phenomenon that turbulent random advection can induce a regularizing effect on the advected solution, even in the absence of viscosity.
Our main findings can be loosely stated as follows (see \cref{thm:main1} and \cref{thm:main2} for the precise mathematical results):

\begin{theorem}[Informal statement]\label{thm:main_informal}
    Let $\alpha\in (0,1)$ fixed. For any $s\in (0,d/2)$, there exist positive constants $K$, $C$ such that
    \begin{equation}\label{eq:intro_main_thm}
		\frac{\dd}{\dd t} \EE\big[\|\rho(t,\cdot) \|_{\dot H^{-s}}^2 \big] + K\, \EE\big[ \|\rho(t,\cdot) \|_{\dot H^{-s+1-\alpha}}^2 \big] \leq C\,  \EE\big[\|\rho(t,\cdot) \|_{\dot H^{-s}}^2\big]\quad \forall\, t\geq 0.
	\end{equation}
    for all solutions $\rho$ to \eqref{eq:intro_kraichnan}.
    In particular:
    \begin{itemize}
        \item[a)] Solutions associated to initial data $\rho_0\in L^2$ gain positive Sobolev regularity and $\PP$-a.s. belong to $L^2([0,T];H^{1-\alpha-})$.
        \item[b)] Solutions associated to initial datum $\rho_0$ in negative Sobolev spaces $\dot H^{-s}$, $0<s<d/2$, are still well-defined and instantaneously become $L^2$-regular at $t>0$; by Part a), such generalized solutions $\PP$-a.s. belong to $L^2([\delta,T];H^{1-\alpha-})$ for any $\delta>0$. 
    \end{itemize}
\end{theorem}

Roughly speaking, for fixed $\alpha\in (0,1)$, by \eqref{eq:intro_main_thm} there is a consistent regularity gain of order $1-\alpha$ across all negative Sobolev norms $\dot H^{-s}$ with $s\in (0,d/2)$. Such spaces are commonly used to gauge mixing \cite{MaMePe2005,Thiffeault2012} and to study the distribution of energy across different Fourier modes, in particular in order to identify cascade phenomena.

\cref{thm:main_informal} can be viewed as an \emph{anomalous regularization} result and the mechanism generating it is the same one responsible for anomalous dissipation \eqref{eq:intro_anomalous_dissipation}.
Indeed, in the case of regular transport equations, the dynamics is reversible and by \eqref{eq:intro_Lagrangian_formula} one has $\rho_0(x)=\rho(t,\Phi(t,x))$; in particular, for sufficiently regular $\Phi$, $\rho(t,\cdot)\in H^s$ if and only if $\rho_0\in H^s$.
In particular, turbulent flows display a symmetry breaking of time reversibility. 
We see again a nice parallel with Burgers' equation, where $L^\infty$ initial data become $BV$ regular at positive times by Oleinik's result \cite{DLOM2004}.
Such self-regularizing mechanisms are expected to hold in more complicated hydrodynamic turbulence, see the discussion in \cite{Drivas2022}; the Onsager threshold $C^{1/3}$ in turbulent solutions of the Euler equations, should not only be informative of limited regularity of the fluid, but also of the presence of internal alignment mechanisms that lead to its saturation.

\begin{remark}[Optimality]\label{rem:optimality}
    We expect \cref{thm:main_informal} to be sharp, in the following sense:
    \begin{itemize}
        \item[i)] Solutions cannot belong to $L^2_t H^{1-\alpha}$, so that we cannot expect a similar balance to be true for $s\leq 0$ (namely, when $\rho_0$ belongs to a space of positive Sobolev regularity). Indeed, a commutator argument shows that, if $\rho\in L^2_t H^{1-\alpha}$, then its $L^2$-norm must be preserved; this was originally communicated to us by T. Drivas and can now be found in \cite[Theorem 1.3 and Section 3.2]{DGP2025} (in a slightly sharper statement, which also covers Besov-type regularity). Since the SPDE \eqref{eq:intro_kraichnan} is known to display anomalous dissipation by the aforementioned references, regularity $L^2_t H^{1-\alpha}$ would yield a contradiction.
        \item[ii)] White noise is formally invariant under \eqref{eq:intro_kraichnan}; arguing as in \cite{FlaLuo2019,GroPec2022} (which treat more complicated nonlinear (S)PDEs), one can construct \emph{white noise solutions} to \eqref{eq:intro_kraichnan}. Such solutions have stationary marginals which belong to $H^{-d/2-}\setminus H^{-d/2}$, providing counterexamples to the same regularising effect for $s>d/2$, at least for this class of solutions. Note that this does not necessarily exclude the existence of strong solutions (obtained by Wiener chaoses) witnessing anomalous regularization; however, if that were the case, pathwise uniqueness and uniqueness in law of solutions couldn't hold, and the dynamics in such poor regularity spaces might be strongly ill-defined.
    \end{itemize}
    We do however expect the result to extend to the borderline value $s=d/2$, in agreement with \cite{coghi2023existence} which already dealt with the case $d=2,\,s=1$.
    Let us also mention that in the borderline case $s=0$, sharp borderline Besov regularity $\tilde L^2_{\omega,t} \tilde B^{1-\alpha}_{2,\infty}$ was recently established in \cite[Theorem 5.1 and Remark 5.2]{DGP2025}.
\end{remark}

It is worth mentioning that much of the Physics literature on Kraichnan's model has focused on the \emph{statistically self-similar} case, that is formally when $Q(0)-Q(z)$ has an exact power-law scaling, obtained by replacing $(1+|\xi|^2)^{d/2+\alpha}$ with $|\xi|^{d+2\alpha}$ in \eqref{eq:kraichnan_noise} (see eq. \eqref{eq:covariance_selfsimilar} below); see for instance \cite{EyinkXin1996,EyiXin2000,BeGaKu1998,MajKra1999,hakulinen2003}.
Additionally, these works sometimes make use of additional assumptions that are difficult to justify with mathematical rigor, or even unclear, such as the existence of invariant measures (equilibrium states) which are the unique long-time limit of all solutions, or initial data being sampled randomly and translation invariant (thus necessarily implying $\EE[\| \rho_0\|_{H^s}]=+\infty$ for all $s\in\R$).

In comparison, the rigorous mathematical literature on Kraichnan model is relatively small and concentrated in the early 2000's, most notably in the works by E-Vanden Eijnden \cite{EVan2000,EVan2001}, LeJan-Raimond \cite{LeJRai2002,LeJRai2004} and Eyink-Xin \cite{EyiXin1996}.
Among recent contributions, let us recall the (already mentioned) exponential mixing results from \cite{GesYar2021} for $\alpha>1$, the recent work \cite{zelati2023statistically} in the self-similar case $\alpha=1$ and finally \cite{Rowan2023}, proving asymptotic exponential decay of energy for rough Kraichnan on $\mathbb{T}^d$.\\

In the statistically self-similar case, even giving mathematical meaning to the SPDE \eqref{eq:intro_kraichnan} becomes non-obvious, as explained in \cref{sec:self-similar}.
We believe one of the contributions of this work is instead to keep all the analysis at the level of \eqref{eq:kraichnan_noise}, where a robust solution theory for the SPDE \eqref{eq:intro_kraichnan} is available, as reviewed in \cref{sec:preliminaries}. This prevents us from performing explicit computations based on self-similar Ans\"atze and requires to work with Mellin transforms, cf. \cref{sec:asymp}.
Nonetheless we note that, at least formally, in the self-similar case our main estimate \eqref{eq:intro_main_thm} actually becomes the exact equality
\begin{equation}\label{eq:intro_selfsimilar_balance}
	\frac{\dd}{\dd t}\EE\big[\|\rho(t,\cdot) \|_{\dot H^{-s}}^2 \big] + K\, \EE\big[ \|\rho(t,\cdot) \|_{\dot H^{-s+1-\alpha}}^2 \big] =0.
\end{equation}
A formal proof of \eqref{eq:intro_selfsimilar_balance} based on our main results will be presented in \cref{sec:self-similar}.\\

Besides the new insight on stochastic passive scalar advection,
the quantitative estimates of \cref{thm:main1} establish a framework for regularization of nonlinear SPDEs by transport noise,
as already initiated in recent contributions \cite{coghi2023existence,GalLuo2023,JiaLuo2024,bagnara2024regularization}; after the appearance of this manuscript, this direction has been further developed in \cite{JiaLuo2025a,JiaLuo2025b,bagnara2026refined}.
In particular, this work can be regarded as a follow up to \cite{coghi2023existence}, where the balance \eqref{eq:intro_main_thm} was first rigorously shown for $s=1$, $d=2$ and applied to prove well-posedness of the stochastic $2$D Euler equations with $L^p$-valued vorticity.

After the first version of this manuscript appeared, several further results concerning anomalous regularization have followed. Specifically, \cite{BGM2024} considers the Kasantzev-Kraichnan model of vector advection, for a suitable range of parameters $\alpha$ and negative Sobolev norms $H^{-s}$ (see \cite[Hypothesis 3.11 and Theorem 3.12]{BGM2024}); instead, \cite{CLP2025} considers a larger class of scalar transport equations with additional deterministic drift $b\cdot\nabla \rho \dd t$, proving well-posedness, anomalous regularization, and the existence of the dissipation measure, as well as characterizing the zero-noise limit and establishing an LDP for it.
\cite{DGP2025} extends the present results to the full weakly compressible regime $\eta>1-d/(4\alpha^2)$ (without assuming $W$ divergence-free) and provides several other quantitative estimates (see \cite[Theorem 1.3]{DGP2025}), as well as counterpart results concerning Richardson's law (\cite[Theorem 1.5]{DGP2025}); in the divergence-free case, it establishes sharp Besov regularity $\tilde L^2_{\omega,t} \tilde B^{1-\alpha}_{2,\infty}$ and a link between this regularity result and the structure of the anomalous dissipation measure (Theorem 1.2).
Finally, the work \cite{Rowan2025} provides a general framework to establish anomalous dissipation and anomalous regularization results on the torus $\mathbb{T}^d$ (\cite[Assumption 1.15 and Theorem 1.16]{Rowan2025}), resulting in upper and lower bounds for the associated invariant measures (\cite[Proposition 1.11 and Theorem 1.13]{Rowan2025}); the latter in particular provide a rigorous justification for \cref{rem:invariant_measures} below, in the non-self-similar setting.
Differently from all the previous references, on $\T^d$ one cannot exploit the isotropy of the noise covariance; the analysis from \cite{Rowan2025} instead relies on suitable weighted Poincaré inequalities in the lattice (\cite[Section 6]{Rowan2025}). 

\subsection*{Outline of the paper}
We conclude this introduction by defining the main notation adopted in the paper.
In \cref{sec:preliminaries}, we provide all necessary preliminary results concerning SPDE \eqref{eq:intro_kraichnan}, in particular solution concepts, well-posedness results and vanishing viscosity approximations.
In \Cref{sec:negsobol} we derive exact identities for the evolution of the $\dot{H}^{-s}$-norms of solutions, which are governed by deterministic flux functions $F$ in Fourier variables. In \Cref{sec:asymp} we obtain pointwise bounds on the flux functions, allowing to complete the proofs of main results (\cref{thm:main1,thm:main2}) in \cref{subsec:mainproofs}.
In \cref{sec:self-similar}, we discuss some (formal) consequences of our results in the statistically self-similar case.
Finally in \cref{app:formal_derivation} we provide a short formal derivation of \eqref{eq:intro_selfsimilar_balance} working directly in real variables.


\subsection*{Notation}
Throughout the paper we adopt the following notation and conventions:
\begin{itemize}[leftmargin=6mm]
    \item We always work on $\R^d$ with $d\geq 2$. For $v\in\R^d$, $|v|$ denotes its Euclidean norm; whenever it doesn't cause confusion, we also adopt the Japanese brackets $\brak{v}=(1+|v|^2)^{1/2}$.
    \item For $v\in \R^d\setminus\{0\}$, we denote by 
\begin{align*}
    P^\perp_v = I_d - \frac{v}{|v|}\otimes\frac{v}{|v|}
\end{align*}
the orthogonal projection matrix on the subspace $v^\perp$,
where $I_d$ stands for the identity matrix on $\R^d$ and $(x\otimes y)_{ij}=x_iy_j$.
The projection $P^\perp_v$ enjoys several useful properties, in particular
\begin{align*}
    (P^\perp_v)^2=P^\perp_v, \quad P^\perp_{\lambda v} = P^\perp_v, \quad P^\perp_v v=0
\end{align*}
and $P^\perp_{Ov} O= O P^\perp_v$ for any isometry $O$. Moreover $|P^\perp_v x|^2 \leq |x|^2$ for every $x\in \R^d$ and
\begin{align*}
    |P_{\xi-\eta}^\perp \xi|^2
    = |P_{\xi-\eta}^\perp \eta|^2
    = |\eta|^2 - \frac{(\eta\cdot (\xi-\eta))^2}{|\xi-\eta|^2},\quad \forall\,\xi,\eta\in\R^d.
\end{align*}
    \item When working with functions on $\R^d$, the bracket $\langle f,g \rangle$ denotes the $L^2$ scalar product $\int_{\R^d} fg \dd x$, or more generally the duality pairing between Schwartz functions $\mathcal{S}(\R^d)$ and tempered distributions $\mathcal{S}'(\R^d)$.
    \item We choose the angular frequency definition of Fourier transform:
    \begin{equation*}
    \hat f(\xi)=(2\pi)^{-\frac{d}{2}}\int_{\R^d} f(x) e^{-i\xi\cdot x} \dd x,\quad
    f(x)=(2\pi)^{-\frac{d}{2}} \int_{\R^d} \hat f(\xi) e^{i\xi\cdot x} \dd\xi;  
\end{equation*}
in particular (denoting by $\overline{\hat{g}}$ the complex conjugate of $\hat{g}$)
\begin{equation}\label{eq:convention_fourier}
    \widehat{f\ast g} = (2\pi)^{\frac{d}{2}} \hat f\,\hat g, \quad
    \langle f,g \rangle = \langle \hat{f},\overline{\hat{g}} \rangle, \quad
    \| f\ast g\|_{L^2} = (2\pi)^{\frac{d}{2}} \| \hat f\,\hat g\|_{L^2}.
\end{equation}
    \item For $s\in \R$, the inhomogeneous Sobolev spaces $H^s(\R^d)$, $H^s$ for short, are defined via Fourier transform by
    \begin{align*}
        H^s = \Big\{f\in \mathcal{S}'(\R^d): \| f\|_{H^s}^2 := \int_{\R^d} \langle \xi\rangle^{2s} |\hat f(\xi)|^2 \dd \xi <\infty\Big\}.
    \end{align*}
    Similarly, the homogeneous Sobolev spaces $\dot H^s$ are defined by replacing $\langle \xi \rangle$ by $|\xi|$ in the definition of $\| \cdot\|_{H^s}$ (resp. $\| \cdot\|_{\dot H^s}$). Note that for $s>0$ it holds $H^s\hookrightarrow \dot H^s$, while for $s<0$ $\dot H^s\hookrightarrow H^s$. We adopt the convention that both spaces coincide with $L^2$ for $s=0$.
    \item The symbol $\one_A$ indicates the indicator function of a set $A$.
    \item We always work with filtered probability spaces $(\Omega,\mathcal{F},(\mathcal{F}_t)_t,\PP)$ satisfying the standard assumptions. Elements of $\Omega$ are denoted by $\omega$ and $\EE$ denotes expectation w.r.t. $\PP$.
    \item To perform estimates, we often work with $t\in [0,T]$, for $T$ arbitrary large but finite. We use shortcuts to denote function spaces of mixed regularity involving the variables $\omega$, $t$ and $x$. For example $L^2_t\dot{H}^{-s}_x$ stands for $L^2([0,T];\dot{H}^{-s}(\R^d))$,
    $L^m_\omega L^\infty_t L^p_x$ for $L^m(\Omega; L^\infty([0,T]; L^p(\R^d)))$, and so on. Whenever several exponents coincide, we further abbreviate the notation, so that e.g. $L^2_{\omega,t,x}$ denotes $L^2(\Omega\times[0,T]\times \R^d)$.
    When the subscript is omitted, the space variable is assumed to be $x\in \R^d$.
    \item From now on, when working with stochastic processes (e.g. $\rho$), we will adopt the standard probabilistic convention to indicate the time evaluation at $t\in [0,T]$ with the subscript $\rho_t$.
    \item Concerning asymptotic relations, Landau's $O$ and $o$ have their standard meaning,
and we denote by $\sim$ asymptotic equivalence, \emph{i.e.} $f(x)\sim g(x)$ iff $f(x)=g(x)(1+o(1))$.
We will also write $f(x)\lesssim g(x)$ if and only if $f(x)=O(g(x))$, and $f(x)\simeq g(x)$ if and only if $f(x)\lesssim g(x)$ and $f(x)\gtrsim g(x)$.
Subscripts of Landau's notation or asymptotic relation symbols denote dependence on parameters.
    \item The volume element $\dd\sigma_{d-1}$ on $S^{d-1}\subset \R^d$ in spherical coordinates 
$\theta_1,\theta_2,\dots,\theta_{d-2}\in [0,\pi]$, $\theta_{d-1}\in [-\pi,\pi]$ is
\begin{equation}
    d\sigma_{d-1}= \sin ^{d-2} \theta_1 \sin ^{d-3} \theta_2 \cdots \sin \theta_{d-2} \dd \theta_1 \dd \theta_2 \cdots \dd \theta_{d-1},
\end{equation}
and we write $\omega_{d-1}=\sigma_{d-1}(S^{d-1})=2\pi^{d/2}/\Gamma(d/2)$.
    \item If $f(z)$ is a meromorphic function of $z\in\C$, we denote by $\res_{z=a}\bra{f(z)}$ its residue at $z=a$, in particular if $a$ is a simple pole,
\begin{equation*}
    \res_{z=a}\bra{f(z)}=\lim_{z\to a} (z-a) f(z).
\end{equation*}
\end{itemize}



\section{Preliminaries}\label{sec:preliminaries}

\subsection{Structure of the noise and stochastic integrals}

In the following, $(\Omega,\mathcal{F},(\mathcal{F}_t)_t,\PP)$ is a filtered probability space satisfying the standard assumptions and $W$ is the centered $\R^d$-valued $(\mathcal{F}_t)_t$-Wiener process with covariance \eqref{eq:covariance_W} where $Q$ is given by \eqref{eq:kraichnan_noise}, for some $\alpha\in (0,2)$. In particular, increments $W_t-W_s$ are independent of $\mathcal{F}_s$ whenever $s<t$.

Correspondingly, we define the convolution operator $\cQ f:= Q\ast f$, which is self-adjoint and positive definite on $L^2(\R^d;\R^d)$, and its square root $\cQ^{1/2}$. The It\^o isometry for stochastic integrals w.r.t. $W$ naturally involves the norm
\begin{align*}
    \| \cQ^{1/2} f\|_{L^2}^2
    = \langle Q\ast f, f\rangle
    = (2\pi)^{d/2} \int_{\R^d} \overline{\hat{f}(\xi)}\cdot \hat{Q}(\xi)\hat{f}(\xi) \dd \xi
\end{align*}
where in the second identity we employed our convention \eqref{eq:convention_fourier}. Noting that \eqref{eq:kraichnan_noise} amounts to
\begin{equation}\label{eq:fourier_Q}
    \hat Q(\xi)= \frac{1}{(1+|\xi|^2)^{d/2+\alpha}} P^\perp_\xi,
\end{equation}
we arrive at the formula
\begin{equation}\label{eq:CM_formula}
    \| \cQ^{1/2} f\|_{L^2}^2
    = (2\pi)^{d/2} \int_{\R^d} \frac{1}{(1+|\xi|^2)^{d/2+\alpha}} |P^\perp_\xi \hat f(\xi)|^2 \dd \xi.
\end{equation}

The following result follows from \cite[Lemma 2.8]{GalLuo2023}.

\begin{lemma}\label{lem:stochastic_integrals}
	For any $(\cF_t)_t$-progressively measurable process $f:\Omega\times [0,T]\times \R^d\to \R^d$ such that $\PP$-a.s. $\int_0^T \| \cQ^{1/2} f_r\|_{L^2}^2 \dd r<\infty$, the stochastic integral $M_t:= \int_0^t \langle f_r, \dd W_r\rangle$ is a continuous real-valued local martingale and it holds
    \begin{align*}
        [M]_t= \int_0^t \| \cQ^{1/2} f_r\|_{L^2}^2 \dd r =\int_0^t  \langle Q\ast f_r,f_r\rangle \dd r
    \end{align*}
    where $[M]$ denotes the quadratic variation process associated to $M$.
\end{lemma}

\begin{remark}\label{rem:properties_Q}
    By \cite[Lemma 2.2]{GalLuo2023}, $\cQ^{1/2}$ is a bounded linear operator from $L^2$ to itself and from $BV$ to $L^2$, where $BV$ denotes the space of (vector valued) signed measures of bounded variation; by interpolation, this implies that $\cQ^{1/2}\in \mathcal{L}(L^p;L^2)$ for any $p\in [1,2]$.
    Therefore by \cref{lem:stochastic_integrals}, $\int_0^t \langle f_r, \dd W_r\rangle$ is a well defined local martingale as soon as $\int_0^T \| f_r\|_{L^p}^2 \dd r<\infty$ $\PP$-a.s.
\end{remark}

To give meaning to the SPDE \eqref{eq:intro_kraichnan}, we need to further extend the definition of stochastic integrals, and in particular make sense of It\^o integrals of the form $\int_0^\cdot \nabla f_r\cdot \dd W_r$ in the sense of distributions.

\begin{lemma}\label{lem:stoch_int_distributions}
    For any progressively measurable $f\in L^2_{\omega,t,x}$ and any $s\in [0,d/2+\alpha)$, it holds
    \begin{equation}\label{eq:stoch_int_distributions}
        \EE\bigg[\sup_{t\in [0,T]} \Big\| \int_0^t \nabla f_r\cdot \dd W_r\Big\|_{H^{-s-1}}^2\bigg] \lesssim_{\alpha,s,d} \EE\bigg[ \int_0^T \| f_r\|_{H^{-s}}^2 \dd r\bigg]
    \end{equation}
    As  a consequence, the map $f\mapsto \int_0^\cdot \nabla f_r\cdot \dd W_r$ uniquely extends to a map from $L^2_{\omega,t} H^{-s}$ to $L^2_\omega C_t H^{-s-1}$.
\end{lemma}

\begin{proof}
    Let us set $N_t = \int_0^t \nabla f_r\cdot \dd W_r$. 
    First notice that, for $f\in L^2_{\omega,t,x}$, since $W$ is divergence-free, for any complex-valued $\varphi\in W^{1,\infty}$ by integration by parts it holds
    \begin{align*}
        \langle N_t, \varphi\rangle = -\int_0^t \langle f_r \overline{\nabla\varphi}, \dd W_r\rangle
    \end{align*}
    where now $f\,\nabla\varphi\in L^2_{\omega,t,x}$ and so the integral is well-defined by \cref{lem:stochastic_integrals}.
    We take $\varphi=(2\pi)^{-d/2}e^{-i\xi\cdot x}$ and apply \cref{lem:stochastic_integrals}: noting that
    \begin{align}\label{eq:Fourier_derivative}
        \widehat{f\xi e^{-i\xi \cdot}}(-\eta) = \hat{f}(\xi-\eta) \xi, \quad \forall \xi,\eta\in \R^d,
    \end{align}
    by Burkholder-Davis-Gundy inequality we obtain
    \begin{align*}
        \EE\Big[\sup_{t\in [0,T]} |\widehat{N_t}(\xi)|^2\Big]
        & \lesssim \int_0^T\EE\bigg[ \int_{\R^d}\overline{\hat{f}_r(\xi-\eta)}\cdot \hat{Q}(-\eta)\hat{f}_r(\xi-\eta) \dd\eta \bigg] \dd r\\
        & = \int_0^T \int_{\R^d} |P^\perp_\eta \xi|^2 (1+|\eta|^2)^{-\frac{d}{2}-\alpha}\, \EE\big[|\hat f_r(\xi-\eta)|^2\big] \dd \eta \dd r\\
        & \leq |\xi|^2 \int_0^T \int_{\R^d} (1+|\eta|^2)^{-\frac{d}{2}-\alpha}\, \EE\big[|\hat f_r(\xi-\eta)|^2\big] \dd \eta \dd r.
    \end{align*}
    As a consequence
    \begin{align*}
        \EE\Big[\sup_{t\in [0,T]} \| N_t\|_{H^{-s-1}}^2\Big]
        & \leq \int_{\R^d} (1+|\xi|^2)^{-s-1} \EE[\sup_{t\in [0,T]} |\widehat{N_t}(\xi)|^2] \dd \xi\\
        & \lesssim \int_0^T \int_{\R^d\times \R^d} (1+|\xi|^2)^{-s} (1+|\eta|^2)^{-\frac{d}{2}-\alpha} \EE[|\hat f_r(\xi-\eta)|^2] \dd \eta \dd \xi \dd r =(\ast).
    \end{align*}
    Let us define the function $g_r(\eta):= (1+|\eta|^2)^{-s} \EE[|\hat f_r(\eta)|^2]$, so that by assumption $g_r\in L^1$ with $\| g_r\|_{L^1}=\EE[\|f_r\|_{H^{-s}}^2]$; then
    \begin{align*}
        (\ast)
        & = \int_0^T \int_{\R^d\times \R^d} \frac{(1+|\xi-\eta|^2)^{s}}{(1+|\xi|^2)^{s} (1+|\eta|^2)^{\frac{d}{2}+\alpha}} g_r(\xi-\eta) \dd \eta \dd \xi \dd r.
    \end{align*}
    Define $D_1:=\{(\xi,\eta): |\xi-\eta|\leq 2 |\xi|\}$, $D_2=D_1^c=\{(\xi,\eta): |\xi-\eta|> 2 |\xi|\}$ and correspondingly set
    \begin{align*}
        I^i_r := \int_{D_i} \frac{(1+|\xi-\eta|^2)^{s}}{(1+|\xi|^2)^{s} (1+|\eta|^2)^{\frac{d}{2}+\alpha}} g_r(\xi-\eta) \dd \eta \dd \xi \dd r.
    \end{align*}
    On $D_1$, $(1+|\xi-\eta|^2)^{s}\lesssim (1+|\xi|^2)^{s}$ and so
    \begin{align*}
        I^1_r \lesssim \int_{\R^{2d}} \langle (1+|\eta|^2)^{-\frac{d}{2}-\alpha} g_r(\xi-\eta) \dd \xi \dd\eta
        = \|(1+|\cdot|^2)^{-\frac{d}{2}-\alpha}\|_{L^1} \| g_r\|_{L^1} \lesssim \EE[\|f_r\|_{H^{-s}}^2];
    \end{align*}
    instead on $D_2$ it holds $|\xi-\eta|\sim |\eta|$ and so
    \begin{align*}
        I^2_r
        & \lesssim \int_{\R^{2d}} (1+|\xi|^2)^{-s} (1+|\eta|^2)^{-\frac{d}{2}-\alpha+s} g_r(\xi-\eta) \dd \xi \dd\eta\\
        &= \langle (1+|\cdot|^2)^{-s}, (1+|\cdot|^2)^{-d/2-\alpha+s}\ast g_r\rangle\\
        & \leq \| g_r\|_{L^1} \| (1+|\cdot|^2)^{-s}\ast (1+|\cdot|^2)^{-d/2-\alpha+s}\|_{L^\infty}
        \lesssim \EE[\|f_r\|_{H^{-s}}^2]
    \end{align*}
    where in the last passage we used the facts that $d/2+\alpha-s>0, (d/2+\alpha-s)+s>d/2$ and Young's convolution inequality.
    Combining the bounds on $I^1$ and $I^2$, we conclude that \eqref{eq:stoch_int_distributions} holds.
\end{proof}

\begin{remark}
    A closer inspection to the proof of \cref{lem:stoch_int_distributions} reveals that we are actually providing a bound on the bracket process $[N]_t$ associated to the $H^{-s-1}$-valued martingale $N_t=\int_0^t \nabla f_r\cdot \dd W_r$.\footnote{For martingales $N$ with values in a Hilbert space $E$, $[N]_t$ is the unique increasing real-valued process such that $\|N_t\|_E^2 - [N]_t$ is still a martingale.} In particular, it holds
    \begin{align*}
        \frac{\dd}{\dd t} [N]_t \lesssim_{\alpha,s,d} \| f_r\|_{H^{-s}}^2.
    \end{align*}
    By standard localisation arguments, the definition of $N_t$ as a local martingale extends to any progressively measurable process $f$ such that $\int_0^T \| f_r\|_{H^{-s}}^2 \dd r<\infty$ $\PP$-a.s.
\end{remark}

\subsection{Notions of Solutions}

The Kraichnan model \eqref{eq:intro_kraichnan} is typically written in Stratonovich form; this is the physical correct choice in order to interpret it as a passive scalar SPDE and derive representation formulas of the form \eqref{eq:intro_Lagrangian_formula}, at least whenever $W$ and $\rho$ are regular enough.
Mathematically, it is convenient to rewrite the system in It\^o form by computing the It\^o--Stratonovich corrector, yielding
\begin{equation}\label{eq:transportnoise}
    \dd \rho_t(x) + \nabla \rho_t(x)\cdot \dd W_t(x) = \frac{c}{2}\Delta \rho_t(x) \dd t
\end{equation}
where $c$ is the constant such that $Q(0)=cI$, namely (by \eqref{eq:fourier_Q})
\begin{align*}
    c= (2\pi)^{-d/2} \,\frac{d-1}{d} \int_{\R^d} \brak{\xi}^{-d-2\alpha} \dd \xi >0.
\end{align*}
It is well known that the two formulations \eqref{eq:intro_kraichnan} and \eqref{eq:transportnoise} are formally equivalent. 
However, when $W$ and $\rho$ are spatially irregular, a rigorous formulation of \eqref{eq:intro_kraichnan} directly in Stratonovich form is missing, see the discussions in \cite[Section 2.2]{GalLuo2023} and \cite[Remark 2.10]{coghi2023existence}; for this reason, from now on we will systematically work directly with the It\^o formulation \eqref{eq:transportnoise}.
%


In the following, we will say that a $\mathcal{S}'(\R^d)$-valued process $f$ is weakly continuous if, for $\PP$-a.e. $\omega$, $t\mapsto f_t(\omega)$ is continuous w.r.t. the $\mathcal{S}'(\R^d)$-topology (that is, $t\mapsto \langle \varphi, f_t(\omega) \rangle$ is continuous for every $\varphi\in \mathcal{S}(\R^d)$).

\begin{definition}\label{defn:Kraichnan}
    Let $s\in [0,d/2)$ and $\rho_0\in \dot{H}^{-s}$.
    We say that $\rho$ is a $\dot{H}^{-s}$-solution to the SPDE \eqref{eq:transportnoise} if it is a $(\cF_t)_t$-progressively measurable, $H^{-s-2}$-valued weakly continuous process, satisfying $\rho \in L^2_{\omega,t}\dot{H}^{-s}$ and such that $\PP$-a.s.
    \begin{equation}\label{eq:def_sol}
        \rho_t = \rho_0 - \int_0^t \nabla \rho_s\cdot \dd W_s + \int_0^t \frac{c}{2} \Delta \rho_s \dd s, \quad \forall t\in [0,T].
    \end{equation}
\end{definition}

\begin{remark}
    Using the fact that $\Delta$ is a bounded linear operator from $\dot H^{-s}$ to $H^{-s-2}$, and applying \cref{lem:stoch_int_distributions}, all the integrals appearing in \eqref{eq:def_sol} define continuous processes in $H^{-s-2}$.

    Let us note that \cref{defn:Kraichnan} is not the only possible definition of solutions, as one can alternatively give meaning to \eqref{eq:transportnoise} in a variational way by testing with smooth functions $\varphi\in C^\infty_c$; the latter method allows to consider processes with values in $L^1_{loc}$, cf. \cite{GalLuo2023}. The two definitions are equivalent as soon as $\rho$ takes values in $L^2$.
\end{remark}

The following result is a consequence of \cite[Propositions 3.3 and 3.4]{GalLuo2023}.
It guarantees the well-posedness of the inviscid SPDE \eqref{eq:transportnoise} in the class of $(L^1\cap L^2)$-valued solutions. 

\begin{theorem}\label{thm:wellposedness_kraichnan}
    For any $\rho_0\in L^1\cap L^2$, strong existence and pathwise uniqueness hold for equation \eqref{eq:transportnoise}, in the class of $\dot{H}^{-s}$ solutions $\rho$ such that, $\PP$-a.s., $\rho$ is weakly continuous in $L^2$ and $\rho\in L^2_t  L^1$. Moreover, the unique solution satisfies $\PP$-a.s., for any $p\in [1,2]$, the pathwise bound
    \begin{equation}\label{eq:Lp_bound}
        \|\rho_t\|_{L^p} \le \|\rho_0\|_{L^p}, \quad \forall t\in [0,T].
    \end{equation}
\end{theorem}

\subsection{Vanishing viscosity approximations}\label{subsec:vanishing_viscosity}

In order to prove our main results, it will be sometimes useful to derive estimates for the \emph{viscous} Kraichnan model, given in Stratonovich form by
\begin{equation}\label{eq:Kraichnan_viscous}
    \dd \rho^\nu_t(x) + \circ \dd W_t(x)\cdot \nabla \rho^\nu_t(x)  = \nu\Delta \rho^\nu_t(x) \dd t
\end{equation}
and then pass to the limit as $\nu\to 0^+$. As before, it is mathematically convenient to consider the It\^o version of the SPDE \eqref{eq:Kraichnan_viscous}, given by
\begin{equation}\label{eq:transportnoise_viscous}
    \dd \rho^\nu_t(x) + \dd W_t(x)\cdot \nabla \rho^\nu_t(x)  = \left(\frac{c}{2}+\nu\right)\Delta \rho^\nu_t(x) \dd t.
\end{equation}
One can define solutions to \eqref{eq:transportnoise_viscous} as in \cref{defn:Kraichnan}; the goal of this section is to verify that the SPDE \eqref{eq:transportnoise_viscous} is well defined and, most importantly, it recovers solutions to the inviscid SPDE \eqref{eq:transportnoise} as $\nu\to 0^+$.

\begin{lemma}\label{lem:viscous_kraichnan}
    For any $\rho_0\in L^1\cap L^2$, there exists a strong solution to \eqref{eq:transportnoise_viscous}, which satisfies the $\PP$-a.s. bounds
    \begin{equation}\label{eq:bounds_viscous_kraichnan}
        \| \rho^\nu_t\|_{L^2}^2+2\nu\int_0^t \| \nabla\rho^\nu_s\|_{L^2}^2 \dd s \leq \|\rho_0\|_{L^2}^2, \quad \|\rho^\nu_t\|_{L^1}\leq \| \rho_0\|_{L^1} \quad \forall\, t\geq 0.
    \end{equation}
    Moreover, pathwise uniqueness holds for \eqref{eq:transportnoise_viscous} in the class of weakly continuous solutions belonging $\PP$-a.s. to $L^2_t L^1_x$.
\end{lemma}

\begin{proof}
    All the statements are taken from \cite{GalLuo2023}; therein they are presented in the (more challenging) case where $\Delta$ is replaced by a fractional Laplacian $-(-\Delta)^{\beta/2}$ for $\beta\in (0,2)$, but the proofs naturally extend to $\beta=2$ as well. Specifically, strong existence and the $L^1$-bound come from \cite[Proposition 3.1]{GalLuo2023}, the $L^2$-bound from Remark 3.2 and pathwise uniqueness from Proposition 3.4 therein.
\end{proof}

The next result is the main goal of this section.

\begin{proposition}\label{prop:vanishing_viscosity_approximation}
    Let $\rho_0\in L^1\cap L^2$ and let $\rho^\nu$, $\rho$ denote the unique $(L^1\cap L^2)$-valued strong solutions respectively to \eqref{eq:transportnoise} and \eqref{eq:transportnoise_viscous}.
    Then for any $T>0$, $s\in (0,d/2)$ and any $m\in [1,\infty)$ it holds
    \begin{equation}\label{eq:vanishing_viscosity_approximation}
        \lim_{\nu\to 0} \EE\Big[ \sup_{t\in [0,T]} \| \rho^\nu_t - \rho_t\|_{\dot H^{-s}}^m\Big] = 0.
    \end{equation}
\end{proposition}

In order to prove it, we need some preliminary lemmas.

\begin{lemma}\label{lem:tail_estimates}
    Let $\rho_0\in L^1\cap L^2$, $\rho^\nu$ the unique associated solution to \eqref{eq:transportnoise_viscous}. Then for any $T>0$ it holds
    \begin{equation}\label{eq:tail_estimates}
        \lim_{R\to\infty} \sup_{\nu\in (0,1)}\, \EE\Big[\,\sup_{t\in [0,T]} \| \rho^\nu_t\, \one_{|x|>R}\|_{L^2}^2\Big] = 0.
    \end{equation}
\end{lemma}

\begin{proof}
    Given the solution $\rho^\nu$, one can check that $|\rho^\nu|^2$ solves in the sense of distributions
    \begin{equation}\label{eq:SPDE_square}
        \dd |\rho^\nu|^2 + \nabla |\rho^\nu|^2 \cdot \dd W
        = 2 \nu\, \rho^\nu \Delta \rho^\nu \dd t + \frac{c}{2} \Delta( |\rho^\nu|^2) \dd t
        = \Big(\nu + \frac{c}{2}\Big) \Delta( |\rho^\nu|^2) \dd t -2\nu |\nabla \rho^\nu|^2 \dd t.
    \end{equation}
    Rigorously, to verify the above claim one should first spatially mollify the noise $W$, perform computations at the level of smooth solutions $\rho^{\nu,\eps}$ and then send $\eps\to 0$; since the arguments are classical, and we are only interested in the resulting \textit{inequalities} (which are stable under weak convergence), let us skip this technical step.

    Let $\varphi\in C^\infty_b(\RR^d)$ be a nonnegative function such that $\varphi(x)\equiv 0$ for $|x|\leq 1/2$, $\varphi(x)\equiv 1$ for $|x|\geq 1$; set $\varphi^R(x):=\varphi(x/R)$. Then by \eqref{eq:SPDE_square} and integration by parts it holds
    \begin{align*}
        \dd \langle |\rho^\nu|^2,\varphi^R\rangle
        = \frac{1}{R} \left<  \nabla\varphi\Big(\frac{\cdot}{R}\Big)|\rho^\nu|^2, \dd W\right> +  \frac{1}{R^2} \Big(\nu + \frac{c}{2}\Big) \left< \Delta\varphi\Big(\frac{\cdot}{R}\Big), |\rho^\nu|^2\right> \dd t -2\nu \langle \varphi^R, |\nabla \rho^\nu|^2 \rangle \dd t
    \end{align*}
    Noting that the last term on the r.h.s. is negative, integrating in time we find
    \begin{align*}
        \sup_{t\leq T} \, &\| \rho^\nu_t\, \one_{|x|>R}\|_{L^2}^2
        \leq \sup_{t\leq T}\, \langle |\rho^\nu_t|^2,\varphi^R\rangle\\
        & \leq \langle |\rho_0|^2,\varphi^R\rangle + \frac{1}{R} \sup_{t\in [0,T]} \Big| \int_0^t \left< \nabla\varphi\Big(\frac{\cdot}{R}\Big)|\rho^\nu_r|^2, \dd W_r\right>\Big|  + \frac{1}{R^2} \Big(\nu + \frac{c}{2}\Big) \int_0^T \left< |\Delta\varphi|\Big(\frac{\cdot}{R}\Big), |\rho^\nu_r|^2\right> \dd r.
    \end{align*}
    We can estimate the stochastic integral by Doob's inequality and apply \cref{lem:stochastic_integrals}, \cref{rem:properties_Q} to find
    \begin{align*}
        \EE\Big[ \sup_{t\in [0,T]} \Big| \int_0^t \left\langle \nabla\varphi\Big(\frac{\cdot}{R}\Big)|\rho^\nu_r|^2, \dd W_r\right\rangle\Big|^2 \Big]
        & \lesssim \EE\Big[  \int_0^T \Big\| \nabla\varphi\Big(\frac{\cdot}{R}\Big)\,|\rho^\nu_r|^2\Big\|_{L^1}^2 \dd r \Bigg]\\
        & \lesssim \| \nabla \varphi\|_{L^\infty}^2\, \EE\Big[  \int_0^T \|\rho^\nu_r\|_{L^2}^4 \dd r \Bigg]
        \lesssim T \| \rho_0\|_{L^2}^4,
    \end{align*}
    where we employed the pathwise bound \eqref{eq:bounds_viscous_kraichnan}.
    Similarly, for the deterministic integral we have the pathwise estimate
    \begin{align*}
        \int_0^T \left< |\Delta\varphi|\Big(\frac{\cdot}{R}\Big), |\rho^\nu_r|^2\right> \dd r
        \lesssim \| \Delta\varphi\|_{L^\infty}\, T\, \| \rho_0\|_{L^2}^2.
    \end{align*}
    Combining everything, noting that the above estimates are uniform in $\nu\in (0,1)$, we obtain
    \begin{align*}
        \sup_{\nu\in (0,1)} \EE\Big[\, \sup_{t\leq T} \| \rho^\nu_t\, \one_{|x|>R}\|_{L^2}^2 \Big]
        \lesssim \langle |\rho_0|^2,\varphi^R\rangle + \frac{T^{1/2}}{R} \| \rho_0\|_{L^2}^2 + \frac{T}{R^2} \| \rho_0\|_{L^2}^2.
    \end{align*}
    The conclusion now follows by taking the limit as $R\to\infty$, since by assumption $\rho_0\in L^2$ and so by dominated convergence $\langle |\rho_0|^2,\varphi^R\rangle\to 0$ as $R\to\infty$.
\end{proof}

\begin{lemma}\label{lem:negative_sobolev}
    Let $s>0$, $f\in \dot H^{-s}$. Then for any $\delta,\,\eps>0$ it holds
    \begin{align*}
        \| f\|_{\dot H^{-s}}^2 \leq \eps^{2\delta}  \| f\|_{\dot H^{-s-\delta}}^2 + \Big( 1+\frac{1}{\eps^2}\Big)^s \| f\|_{H^{-s}}^2
    \end{align*}
\end{lemma}

\begin{proof}
    It holds
    \begin{align*}
        \| f\|_{\dot H^{-s}}^2
        & 
        = \int_{|\xi|\leq \eps} |\xi|^{-2s} |\hat f(\xi)|^2 \dd \xi + \int_{|\xi|>\eps} |\xi|^{-2s} |\hat f(\xi)|^2 \dd \xi\\
        & \leq \eps^{2\delta} \int_{|\xi|\leq \eps} |\xi|^{-2s-2\delta} |\hat f(\xi)|^2 \dd \xi + \Big( 1+\frac{1}{\eps^2}\Big)^s \int_{|\xi|>\eps} \langle \xi\rangle^{-2s} |\hat f(\xi)|^2 \dd \xi
    \end{align*}
    where we used the fact that $|y|^{-2} \leq (1+\eps^{-2}) (1+|y|^2)^{-1}$ whenever $|y|\geq \eps$.
    By the definitions of the $\dot H^{-s-\delta}$- and $H^{-s}$-norms, conclusion follows.    
\end{proof}

With this preparations we can complete the

\begin{proof}[Proof of \cref{prop:vanishing_viscosity_approximation}]
    We divide the proof in two main steps.

    \emph{Step 1: convergence in $H^{-s}_w$.} Given the smooth weight $w(x)=(1+|x|^2)^{-d/2-1}$, let us denote by $H^{-s}_w$ the associated weighted negative Sobolev spaces, with norm $\| f\|_{H^{-s}_w} := \| f\, w\|_{H^{-s}}$. Their usefulness comes from the embedding $H^{-s}\hookrightarrow H^{-s}_w$ and compact one $H^{-s}\hookrightarrow H^{-s-\eps}_w$, cf. \cite[Lemma A.4]{bagnara2024regularization}. We first claim that, for any $T>0$, $\delta>0$ and any $m\in [1,\infty)$, it holds
    \begin{equation}\label{eq:vanishing_viscosity_claim}
        \lim_{\nu\to 0} \EE\Big[ \sup_{t\in [0,T]} \| \rho^\nu_t - \rho_t\|_{H^{-s}_w}^m\Big] = 0.
    \end{equation}
    The argument proving \eqref{eq:vanishing_viscosity_claim} is standard, based on tightness estimates and the Gyongy-Krylov lemma \cite[Lemma 1.1]{GyoKry1996}; in a nutshell, we need to verify that $\{\rho^\nu-\rho\}_{\nu>0}$ is tight in $C^0_t H^{-s}_w$ and that any limit point is necessarily $0$.
    As we haven't found a direct reference in the literature, we shortly sketch the proof.

    By \eqref{eq:Lp_bound}-\eqref{eq:bounds_viscous_kraichnan}, we know that $\rho^\nu-\rho$ are uniformly bounded in $L^\infty_{\omega,t} (L^1\cap L^2)$. To verify uniform continuity in time, let us focus on $\rho^\nu$ with \eqref{eq:transportnoise_viscous} written in integral form, the bounds for $\rho$ being similar. Again by \eqref{eq:bounds_viscous_kraichnan}, we have the $\PP$-a.s. bound
    \begin{align*}
        \Big\| \int_s^t \Big( \frac{c}{2} + \nu\Big) \Delta\rho^\nu_r \dd r\Big\|_{H^{-2}}
        \leq \Big( \frac{c}{2} + \nu\Big) \int_s^t \| \rho^\nu_r\|_{L^2} \dd r \lesssim |t-s| \| \rho_0\|_{L^2}\quad \forall\, s<t;
    \end{align*}
    instead for the stochastic integral, applying \cite[Lemma 2.8]{GalLuo2023} and exploiting the fact that $W$ is divergence-free, for any $m\in [1,\infty)$ and $s<t$ it holds
    \begin{align*}
        \EE\bigg[\Big\| \int_s^t \nabla\rho^\nu_r\cdot\dd W_r \Big\|_{H^{-1}}^m\bigg]
        & \leq \EE\bigg[\Big\| \int_s^t \rho^\nu_r\cdot\dd W_r \Big\|_{L^2}^m\bigg]\\
        & \lesssim {\rm Tr}(Q(0))\, \EE\Big[ \Big( \int_s^t \|\rho^\nu_r\|_{L^2}^2 \dd r\Big)^{m/2} \Big]
        \lesssim \| \rho_0\|_{L^2}^m |t-s|^{m/2}.
    \end{align*}
    By an application of Kolmogorov's continuity theorem, we deduce that for any $\gamma<1/2$
    \begin{align*}
        \EE\big[ \| \rho^\nu-\rho\|_{L^\infty_t (L^1\cap L^2)}^m + \| \rho^\nu-\rho\|_{C^\gamma_t H^{-2}}\big] \lesssim \| \rho_0\|_{L^2}^m;
    \end{align*}
    by \cite[Lemma 3.5]{bagnara2024regularization}, $\{\rho^\nu-\rho\}_{\nu>0}$ is tight in $C^0_t H^{-s}_w$, for any $s>0$.

    By Prokhorov's and Skorokhod's theorems, we can then extract a subsequence and pass to a new probability space (both not relabelled for simplicity) such that $\rho^\nu-\rho$ converges in $C^0_t H^{-\delta}_w$ to some limit $\tilde \rho$. Exploiting the linearity of the SPDEs \eqref{eq:transportnoise}-\eqref{eq:transportnoise_viscous}, performing standard arguments one can show that $\tilde\rho$ is a weak solution to the SPDE \eqref{eq:transportnoise} (driven by a new Kraichnan noise $\tilde W$) with initial condition $\tilde \rho_0=0$.
    On the other hand, by properties of weak convergence, for any $t\geq 0$ we have
    \begin{align*}
        \| \tilde\rho_t\|_{L^1\cap L^2} \leq \liminf_{\nu\to 0} \| \rho^\nu_t-\rho_t\|_{L^1\cap L^2} \leq 2\| \rho_0\|_{L^1\cap L^2},
    \end{align*} 
    so that in particular $\tilde\rho \in L^\infty_{\omega,t} L^1$; but then by \cref{thm:wellposedness_kraichnan} it necessarily holds $\tilde\rho\equiv 0$. As the argument holds for any subsequence we can extract, claim \eqref{eq:vanishing_viscosity_claim} follows.

    \emph{Step 2: Conclusion.} We first upgrade \eqref{eq:vanishing_viscosity_claim} by showing that
    \begin{equation}\label{eq:vanishing_viscosity_claim2}
        \lim_{\nu\to 0} \EE\Big[ \sup_{t\in [0,T]} \| \rho^\nu_t - \rho_t\|_{H^{-\delta}}^m\Big] = 0.
    \end{equation}
    Let $\varphi$, $\varphi_R$ be cutoff functions as defined in \cref{lem:tail_estimates}, so that $1-\varphi^R$ is smooth and compactly supported.
    By \eqref{eq:vanishing_viscosity_claim}, $\rho^\nu (1-\varphi^R)\to \rho^\nu (1-\varphi^R)$ in $C_t H^{-\delta}$; on the other hand, by \eqref{eq:tail_estimates}, $\rho^\nu \varphi^R$ can be made arbitrarily small in $H^{-s}$, since
    \begin{align*}
        \sup_{t\in [0,T]} \| \rho^\nu \varphi^R\|_{H^{-s}} \leq \sup_{t\in [0,T]} \| \rho^\nu \varphi^R\|_{L^2}.
    \end{align*}
    Combining these facts together, convergence \eqref{eq:vanishing_viscosity_claim2} follows.
    It remains to further improve this by replacing the $H^{-s}$-topology with the $\dot H^{-s}$-one.
    To this end, fix any $s\in (0,d/2)$ and any $\delta>0$ such that $s+\delta<d/2$. By the embedding $L^1\cap L^2\hookrightarrow \dot H^{-s-\delta}$, the estimates \eqref{eq:Lp_bound}-\eqref{eq:bounds_viscous_kraichnan} and \cref{lem:negative_sobolev}, for any $\eps>0$ it holds
    \begin{align*}
        \sup_{t\in [0,T]} \| \rho^\nu_t-\rho_t\|_{\dot H^{-s}}^m
        \lesssim \eps^{m \delta}  \|\rho_0\|_{L^1\cap L^2}^m + C_\eps \sup_{t\in [0,T]} \| \rho^\nu_t-\rho_t\|_{H^{-s}}^m.
    \end{align*}
    Taking expectation and then in order the limits $\nu\to 0^+$ and $\eps\to 0^+$, \eqref{eq:vanishing_viscosity_approximation} follows from \eqref{eq:vanishing_viscosity_claim2}.
\end{proof}


\section{Evolution of negative Sobolev norms in Kraichnan's model}\label{sec:negsobol}


The following lemma, based on \cite{GalLuo2023}, describes the mean evolution of quantities associated to the energy spectrum $|\hat\rho(\xi)|^2$ of a given solution $\rho$.

\begin{lemma}\label{lem:evol_regular_quantities}
    Let $\rho_0\in L^1\cap L^2$ and let $\rho$ be the unique solution to \eqref{eq:transportnoise} given by \cref{thm:wellposedness_kraichnan}. Let $\psi:\R^d\to\R$ be an integrable function with compact support.
    Then the function
    \begin{equation}\label{eq:defn_a}
        a_t(\xi):=\EE[|\hat\rho_t(\xi)|^2]
    \end{equation}
    satisfies
    \begin{equation}\label{eq:evolution_observables}
	\frac{\dd}{\dd t } \int_{\R^d} a_t(\xi) \psi(\xi) \dd \xi
	= (2\pi)^{-d/2} \int_{\R^d\times \R^d} a_t(\xi) |P^\perp_{\xi-\eta}\xi|^2 \frac{1}{\brak{ \xi-\eta}^{d+2\alpha}} (\psi(\eta)-\psi(\xi)) \dd \xi \dd \eta.
\end{equation}
\end{lemma}

\begin{proof}
    The argument is taken from \cite[Proposition 3.4 and Remark 3.5]{GalLuo2023}, we give only a sketch.
    Notice that, since $\rho_t\in L^1$, $\hat\rho_t(\xi)$ is a pointwise defined, continuous function.
    Passing to Fourier transform (that is, testing with $(2\pi)^{-d/2}e^{-i\xi\cdot x}$), $\hat{\rho}(\xi)$ satisfies, for every $\xi$,
    \begin{align*}
        \dd \hat\rho(\xi) = -(2\pi)^{-d/2}i\langle \xi e^{-i\xi \cdot}\rho, \dd W(\xi) \rangle -\frac{c}{2}|\xi|^2 \hat\rho(\xi) \dd t.
    \end{align*}
    We apply It\^o formula and recall \cref{lem:stochastic_integrals}, formulas \eqref{eq:CM_formula}-\eqref{eq:Fourier_derivative} and the fact that $Q(0)=cI$, and get the equation for $|\rho(\xi)|^2$ for every $\xi$:
    \begin{align*}
        \dd |\hat\rho(\xi)|^2 &= \dd M^\xi -c|\xi|^2 |\hat\rho(\xi)|^2 \dd t +(2\pi)^{-d/2}\int_{\R^d} \overline{\widehat{\xi e^{-i\xi \cdot}\rho\ }(\eta)} \cdot \hat{Q}(\eta) \widehat{\xi e^{-i\xi \cdot}\rho\ }(\eta) \dd \eta \dd t\\
        &=\dd M^\xi -\xi \cdot Q(0)\xi |\hat\rho(\xi)|^2 \dd t +(2\pi)^{-d/2}\int_{\R^d}\xi \cdot \hat{Q}(\eta)\xi\, |\hat{\rho}(\xi+\eta)|^2 \dd \eta \dd t\\
        &=\dd M^\xi +(2\pi)^{-d/2}\int_{\R^d}\xi \cdot \hat{Q}(\eta)\xi\, \big(|\hat{\rho}(\xi+\eta)|^2-|\hat{\rho}(\xi)|^2\big) \dd \eta \dd t,
    \end{align*}
    where $M^\xi$ is a martingale. Taking expectation and using the expression \eqref{eq:fourier_Q} for $\hat{Q}$, after the change of variables $\tilde\eta =\xi+\eta$ we obtain
    \begin{align}\label{eq:spectrum}
        \frac{\dd}{\dd t}a(\xi) = (2\pi)^{-d/2}\int_{\R^d} \frac{1}{\brak{\xi-\eta}^{d+2\alpha}} |P^\perp_{\xi-\eta}\xi|^2 (a(\eta)-a(\xi)) \dd \eta.
    \end{align}
    Since $\rho\in L^\infty_{t,\omega} L^1$, $a$ is bounded, so we can test \eqref{eq:spectrum} with any integrable function $\psi$ with compact support, getting \eqref{eq:evolution_observables}.
\end{proof}

Formally taking $\psi(\xi) = |\xi|^{-2s}$ in \cref{lem:evol_regular_quantities}, we obtain an equation for the evolution of the $\dot{H}^{-s}$ norm of $\rho$. The next result rigorously justifies this procedure for $s>1-\alpha$. 

\begin{proposition}\label{prop:identity_neg_sobolev}
	Let $\alpha\in (0,1)$ and $W$ given by \eqref{eq:kraichnan_noise}. Let $s\in (1-\alpha,d/2)$, $\rho_0\in L^1\cap L^2$ and $\rho$ be the associated unique solution to \eqref{eq:transportnoise}. Then we have
	\begin{equation}\label{eq:fourier-balance}
    \frac{\dd}{\dd t} \EE\big[\| \rho_t\|_{\dot H^{-s}}^2\big]
    = \int_{\R^d} F(\xi)\, \EE\big[|\hat{\rho}_t(\xi)|^2\big] \dd \xi
\end{equation}
	for the \emph{flux function} $F(\xi)=F(\xi,s,\alpha)$ defined by
	\begin{equation}\label{eq:functions_Fn}
        F(\xi):= (2\pi)^{-d/2} \int_{\R^d} \frac{1}{\brak{ \xi-\eta}^{d+2\alpha}} |P_{\xi-\eta}^\perp \xi|^2
    \Big(\frac{1}{|\eta|^{2s}}-\frac{1}{|\xi|^{2s}}\Big) \dd \eta.
\end{equation}
\end{proposition}


\begin{proof}
    As mentioned above, \eqref{eq:fourier-balance}-\eqref{eq:functions_Fn} formally amounts to \eqref{eq:evolution_observables} for the choice $\psi(\xi)=|\xi|^{-2s}$, which however has not compact support; in order to prove the claim, we will argue by approximation, introducing suitable cutoffs $\psi^n$, applying \eqref{eq:evolution_observables} to them and then passing to the limit.
    
    Let $\varphi\in C^\infty_c(\R_+,\R_+)$ be a decreasing function such that $\varphi\equiv 1$ on $[0,1]$, $\varphi\equiv 0$ on $[2,+\infty)$ and set $\varphi_n(r):=\varphi(r/n)$.
    Define
    \begin{align}\label{eq:F_n}
    F^n(\xi) := (2\pi)^{-d/2} \int_{\R^d}  |P^\perp_{\xi-\eta}\xi|^2 \frac{1}{\brak{ \xi-\eta}^{d+2\alpha}} \bigg(\frac{\varphi_n(|\eta|)}{|\eta|^{2s}} - \frac{\varphi_n(|\xi|)}{|\xi|^{2s}} \bigg) \dd \eta.
\end{align}
    Then we can rigorously apply \eqref{eq:evolution_observables} (integrated in time), for $\psi_n(\xi):= \varphi_n(|\xi|)/|\xi|^{2s}$, to find
    \begin{equation}\label{eq:approximate_balance}
	\int_{\R^d} a_t(\xi) \psi_n(\xi) \dd \xi = \int_{\R^d} a_0(\xi) \psi_n(\xi) \dd \xi +  \int_0^t \int_{\R^{2d}} a_s(\xi) F^n(\xi) \dd \xi \dd s.
    \end{equation}
    By the assumptions on $\rho_0$ and monotone convergence, we have
    \begin{equation}\label{eq:conv_initial_data}
        \lim_{n\to\infty} \int_{\R^d} a_0(\xi) \psi_n(\xi) \dd \xi = \int_{\R^d} a_0(\xi) |\xi|^{-2s} \dd \xi = \| \rho_0\|_{H^{-s}}^2.
    \end{equation}
    In order to obtain \eqref{eq:fourier-balance}-\eqref{eq:functions_Fn} (or rather, their equivalent integrated-in-time version), it suffices to prove that for any $t>0$ it holds
    \begin{equation}\label{eq:fourier_balance_goal}
        \lim_{n\to\infty} \int_0^t \int_{\R^{2d}} a_s(\xi) F^n(\xi) \dd \xi \dd s
        = \int_0^t \int_{\R^{2d}} a_s(\xi) F(\xi) \dd \xi \dd s;
    \end{equation}
    indeed, the conclusion will then follow from combining \eqref{eq:approximate_balance}, \eqref{eq:conv_initial_data} and \eqref{eq:fourier_balance_goal} and applying monotone convergence to the l.h.s. of \eqref{eq:approximate_balance}.

    Since $\rho_0\in L^1\cap L^2$, by \cref{thm:wellposedness_kraichnan} the associated solution $\rho$ to \eqref{eq:transportnoise} satisfies $\| \rho_t\|_{L^1\cap L^2} \leq \| \rho_0\|_{L^1\cap L^2}$.
    By properties of the Fourier transform, $\hat\rho_t\in L^2\cap L^\infty$, so that $a$ as defined in \eqref{eq:defn_a} belongs to $L^\infty_t (L^1_x\cap L^\infty_x)$.
    By dominated convergence, in order to obtain \eqref{eq:fourier_balance_goal}, it is enough to show that: a) $F^n(\xi)\to F(\xi)$ for every $\xi\in \R^d$; b) $|F^n(\xi)|\leq C + g(\xi)$ for some finite constant $C>0$ and some $g\in L^1$.
    We focus on establishing b), as the convergence in a) is an easy consequence of the same bounds derived therein (again by dominated convergence).    
    We distinguish two cases.

    \textbf{Case 1: $s\in [1,d/2)$.} Observing that
    \[|P^\perp_{\xi-\eta} \xi|^2	= |P^\perp_{\xi-\eta}\eta|^2 \leq \min\{|\xi|^2,|\eta|^2\},\]
    we have the basic estimate
    \begin{align*}
        |P^\perp_{\xi-\eta} \xi|^2 \frac{1}{\brak{ \xi-\eta}^{d+2\alpha}} \Big|\frac{\varphi_n(\eta)}{|\eta|^{2s}}-\frac{\varphi_n(\xi)}{|\xi|^{2s}}\Big|
	& \leq \frac{1}{\brak{ \xi-\eta}^{d+2\alpha}} \Big( \frac{1}{|\eta|^{2s-2}} + \frac{1}{|\xi|^{2s-2}}\Big)
    \end{align*}
    uniformly in $n$. Under our conditions on $s$ and $\alpha$, $\brak{ \cdot}^{-d-2\alpha}\in L^1\cap L^\infty$ and $|\cdot|^{2-2s}=h_1+h_2$ with $h_1\in L^1$, $h_2\in L^\infty$; therefore by Young's convolution inequality
    \begin{align*}
        \sup_{\xi\in \R^d} \int_{\R^{d}} \frac{1}{\brak{ \xi-\eta}^{d+2\alpha}} \frac{1}{|\eta|^{2s-2}} \dd \eta
        \leq \|\brak{ \cdot}^{-d-2\alpha}\|_{L^\infty} \| h_1\|_{L^1} + \|\brak{ \cdot}^{-d-2\alpha}\|_{L^1} \| h_2\|_{L^\infty} <\infty.
    \end{align*}
    A similar bound holds with $|\eta|^{2-2s}$ replaced by $|\xi|^{2-2s}$: since $s>1$, we have
    \begin{align*}
        \int_{\R^d} \frac{1}{\brak{ \xi-\eta}^{d+2\alpha}} \frac{1}{|\xi|^{2s-2}} \dd \eta \lesssim \frac{1}{|\xi|^{2s-2}} \lesssim 1 + \frac{1}{|\xi|^{2s-2}} \one_{|\xi|\leq 1}
    \end{align*}
    where $g(\xi):=|\xi|^{2-2s} \one_{|\xi|\leq 1}\in L^1$. Overall, we obtain the conclusion in this case.

    \textbf{Case 2: $s\in (1-\alpha,1)$.}
    First observe that, for $|\xi|\le 1$, we can argue similarly as in Case 1 to find
    \begin{align*}
        |F^n(\xi)|
        \lesssim |\xi|^2 \int_{\R^d}\frac{1}{\brak{ \xi-\eta}^{d+2\alpha}|\eta|^{2s}} \dd \eta +|\xi|^{2-2s} \int_{\R^d}\frac{1}{\brak{ \xi-\eta}^{d+2\alpha}} \dd \eta
    \lesssim 1,
    \end{align*}
    uniformly in $n$; here we used the fact that $|\eta|^{-2s}\in L^1+L^\infty$ since $s<1$ and $d\geq 2$.

    For fixed $\xi$ with $|\xi|> 1$, it is convenient to split $\R^d$ in three subdomains, given by
    \begin{align*}
        D_1 & =\{\eta\in \R^d: |\eta-\xi|\wedge |\eta|> |\xi|/2\}\\
	D_2 & =\{\eta\in \R^d: |\eta-\xi|>|\xi|/2, \ |\eta|\leq |\xi|/2\}\\
	D_3 & =\{\eta\in \R^d: |\eta-\xi|\leq |\xi|/2\}.
    \end{align*}
    Correspondingly, we set $F^n(\xi)=F^{n,1}(\xi)+F^{n,2}(\xi)+F^{n,3}(\xi)$, where $F^{n,i}$ is the function defined similarly to \eqref{eq:F_n}, but with integral restricted to $\eta\in D_i$, $i=1,2,3$.

    Observing that on $D_1$ it holds $|\eta|>|\xi-\eta|/3$, we have
    \begin{align*}
        |F^{n,1}(\xi)|
        &\lesssim \int_{D_1} \frac{|\xi|^2}{\brak{ \xi-\eta}^{d+2\alpha}|\eta|^{2s}} \dd \eta +\int_{D_1}\frac{|\xi|^{2-2s}}{\brak{ \xi-\eta}^{d+2\alpha}} \dd \eta \\
        &\lesssim \int_{D_1} \frac{1}{\brak{ \xi-\eta}^{d+2\alpha-2+2s}} \dd \eta
        \leq \int_{\R^d} \frac{1}{\brak{\eta'}^{d+2\alpha-2+2s}} \dd \eta'
    \end{align*}
where the last integral is finite, due to the condition $s>1-\alpha$.

    On $D_2$ instead it holds
    \begin{align*}
    |F^{n,2}(\xi)|
        &\lesssim \int_{D_2} \frac{|\xi|^2}{\brak{ \xi-\eta}^{d+2\alpha}|\eta|^{2s}} \dd \eta +\int_{D_2}\frac{|\xi|^{2-2s}}{\brak{ \xi-\eta}^{d+2\alpha}} \dd \eta \\
        &\lesssim \int_{D_2} \frac{1}{\brak{ \xi-\eta}^{d+2\alpha-2}|\eta|^{2s}} \dd \eta +\int_{D_2}\frac{1}{\brak{ \xi-\eta}^{d+2\alpha+2s-2}} \dd \eta \\
        & \lesssim \frac{1}{|\xi|^{d+2\alpha-2}} \int_{|\eta|\le |\xi|/2} \frac{1}{|\eta|^{2s}}  \dd\eta
    +\int_{\R^d}\frac{1}{\brak{ \eta'}^{d+2\alpha+2s-2}} \dd \eta'\\
        &\lesssim \frac{1}{|\xi|^{2\alpha+2s-2}} + 1 \lesssim 1
    \end{align*}
    where the last passage relied on $|\xi|>1$ and $\alpha+s-1>0$.

    On $D_3$, we can expand $\psi_n(\eta)=\varphi_n(|\eta|)|\eta|^{-2s}$ around $\xi$ and apply mean-value theorem to find:
    \begin{align}
    \begin{aligned}\label{eq:F_n1_expansion}
        & \left|F^{n,3}(\xi) - (2\pi)^{-d/2} \int_{D_3}  |P^\perp_{\xi-\eta}\xi|^2 \frac{1}{\brak{ \xi-\eta}^{d+2\alpha}} \nabla \psi_n(\xi) \cdot (\eta-\xi) \dd \eta \right|\\
        &\qquad \lesssim
        \int_{D_3}  |P^\perp_{\xi-\eta}\xi|^2 \frac{1}{\brak{ \xi-\eta}^{d+2\alpha}} \sup_{z\in J_\xi}|D^2 \psi_n(z)| |\eta-\xi|^2 \dd \eta=: \tilde F^n(\xi)
    \end{aligned}
    \end{align}
    where the domain $J_\xi$ is defined by
    \begin{align*}
        J_\xi:=\big\{z=\xi + \theta(\eta-\xi) \text{ for some }\theta\in (0,1) \text{ and } \eta\in D_3\big\}.
    \end{align*}
    Note that
    \begin{align*}
        (\eta-\xi)\mapsto |P^\perp_{\xi-\eta}\xi|^2 \frac{1}{\brak{ \xi-\eta}^{d+2\alpha}} \nabla \psi_n(\xi) \cdot (\eta-\xi)
    \end{align*}
    is an odd function (for fixed $\xi$) and the domain $D_3$ is symmetric in $\eta-\xi$, hence the integral appearing on the left-hand side of the expansion \eqref{eq:F_n1_expansion} is zero:
    \begin{align*}
        \int_{D_3}  |P^\perp_{\xi-\eta}\xi|^2 \frac{1}{\brak{ \xi-\eta}^{d+2\alpha}} \nabla \psi_n(\xi) \cdot (\eta-\xi) \dd \eta =0,
        \quad |F^{n,3}(\xi)| \leq \tilde F^n(\xi).
    \end{align*}
    By construction we have
    \begin{align*}
        |\varphi_n'(r)|\lesssim \frac{1}{n} \one_{n\le r\le 2n} \lesssim \frac{1}{r},\quad |\varphi_n''(r)|\lesssim \frac{1}{n^2} \one_{n\le r\le 2n} \lesssim \frac{1}{r^2};
    \end{align*}
    moreover, since $|\eta-\xi|\leq |\xi|/2$, $|z|\ge |\xi|/2$ for every $z\in J_\xi$.
    Hence, by lengthy but elementary computations, for any $z\in J_\xi$ we get
    \begin{align*}
        |D^2\psi_n(z)| \lesssim \frac{1}{|z|^{2s+2}} +\frac{1}{n}1_{n\le |z|\le 2n}\frac{1}{|z|^{2s+1}} +\frac{1}{n^2}1_{n\le |z|\le 2n}\frac{1}{|z|^{2s}}
    \lesssim \frac{1}{|\xi|^{2s+2}}.
    \end{align*}
    Combining these facts, we obtain
    \begin{align*}
        \tilde F^n(\xi)
        \lesssim \frac{1}{|\xi|^{2s}}\int_{D_3} \frac{|\eta-\xi|^2}{\brak{ \xi-\eta}^{d+2\alpha}} \dd \eta
        \lesssim \frac{1}{|\xi|^{2s}} \int_{|\eta'|<|\xi|/2} \frac{1}{| \eta'|^{d+2\alpha-2}} \dd \eta
        \lesssim |\xi|^{2-2s-2\alpha} \lesssim 1
    \end{align*}
    where the last passage comes again from $|\xi|>1$ and $\alpha+s>1$.

    Putting everything together, we conclude that the functions $F^n$ are uniformly bounded by a finite constant, which yields the conclusion.
\end{proof}

The procedure presented in \cref{prop:identity_neg_sobolev} is quite robust and extends to other contexts, as it exclusively relies on the knowledge that $\rho$ is an $L^p$-valued weak solution to \eqref{eq:transportnoise}. In particular it allows for the presence of additional, nonlinear terms in the SPDE, and it has been applied successfully in other contexts to show restoration of well-posedness by transport noise, cf. \cite{coghi2023existence,bagnara2024regularization}.

On the other hand, \cref{prop:identity_neg_sobolev} does not cover the values $s\in (0,1-\alpha]$.
To deal with this case, we employ a different strategy, based on the available knowledge that the solutions to our SPDE are unique and recovered by their viscous approximations, cf. \cref{subsec:vanishing_viscosity}.
In particular, we first pass to consider the viscous equation \eqref{eq:transportnoise_viscous}; thanks to the additional regularity bound \eqref{eq:bounds_viscous_kraichnan}, the rigorous justification of Sobolev balances here is more straightforward.
This will allow us to recover bounds on solutions $\rho$ to the inviscid SPDE by taking the limit $\nu\to 0$, see the proof of \cref{thm:main1} further below.

The next statement is the analogue of \cref{prop:identity_neg_sobolev} for the viscous SPDE:

\begin{proposition}\label{prop:identity_neg_sobolev_general}
    Let $\alpha\in (0,1)$, $\nu>0$ and $W$ given by \eqref{eq:kraichnan_noise}. Let $s\in (0,d/2)$, $\rho_0\in L^1\cap L^2$ and $\rho^\nu$ be the associated solution to equation \eqref{eq:transportnoise_viscous}. Then we have
	\begin{equation}\label{eq:fourier-balance-viscous}
    \frac{\dd}{\dd t} \EE\big[\| \rho^\nu_t\|_{\dot H^{-s}}^2\big]
    = \int_{\R^d} F(\xi)\, \EE\big[|\hat{\rho}^\nu_t(\xi)|^2\big] \dd \xi - 2\nu\, \EE\big[\| \rho^\nu_t\|_{\dot H^{-s+1}}^2\big]
\end{equation}
	for the \emph{flux function} $F(\xi)=F(\xi,s,\alpha)$ defined as in \eqref{eq:functions_Fn}.
\end{proposition}

\begin{proof}
    Arguing as in \cref{lem:evol_regular_quantities}, setting $a^\nu_t(\xi):=\EE[|\hat\rho^\nu(\xi)|^2]$, for any sufficiently regular $\psi:\R^d\to\R$ it holds
    \begin{equation*}
	\frac{\dd}{\dd t } \int_{\R^d} a^\nu_t(\xi) \psi(\xi) \dd \xi
	= (2\pi)^{-d/2} \int_{\R^d\times \R^d} a^\nu_t(\xi) |P^\perp_{\xi-\eta}\xi|^2 \frac{1}{\brak{ \xi-\eta}^{d+2\alpha}} (\psi(\eta)-\psi(\xi)) \dd \xi \dd \eta -2\nu \int_{\R^d} a^\nu_t(\xi) |\xi|^2 \psi(\xi) \dd \xi.
\end{equation*}
    We can then argue as in \cref{prop:identity_neg_sobolev}, approximating $\psi(\xi)=|\xi|^{-2s}$ by suitable cutoffs $\psi^n$ and trying to pass to the limit as $n\to\infty$.
    In the regime $s\in (1-\alpha,d/2)$, the proof proceeds identically as therein, so we only need to consider the case $s\in (0,1-\alpha]$.

    As $\nu>0$, the dominated convergence argument is much simpler thanks to the bounds \eqref{eq:bounds_viscous_kraichnan}, ensuring that $(1+|\xi|^2) a_t(\xi)\in L^1_{t,\xi}$. In particular, it holds
    \begin{align*}
        \Big| a^\nu_t(\xi) |P^\perp_{\xi-\eta}\xi|^2 \frac{1}{\brak{ \xi-\eta}^{d+2\alpha}} (\psi^n(\eta)-\psi^n(\xi))\Big|
        & \leq a^\nu_t(\xi) |\xi|^2 \frac{1}{\brak{ \xi-\eta}^{d+2\alpha}} \Big( \frac{1}{|\eta|^{2s}} + \frac{1}{|\xi|^{2s}}\Big)\\
        & \leq a^\nu_t(\xi) |\xi|^2 \frac{1}{\brak{ \xi-\eta}^{d+2\alpha}} \frac{1}{|\eta|^{2s}} + a^\nu_t(\xi) |\xi|^{2-2s} \frac{1}{\brak{ \xi-\eta}^{d+2\alpha}}
    \end{align*}
    where the last term is integrable over $\R^d\times\R^d$ (again by the estimates \eqref{eq:bounds_viscous_kraichnan} and properties of convolutions) and independent of $n$.
    We can therefore apply dominated convergence to obtain the conclusion.
\end{proof}


\section{Asymptotic Behavior of Flux Functions}\label{sec:asymp}

Having established \cref{prop:identity_neg_sobolev,prop:identity_neg_sobolev_general}, the problem of understanding the evolution of negative Sobolev norms for $\rho$
is translated into the purely analytic problem of establishing estimates for $F(\xi)$.
We will need in particular to understand the leading order terms in the asymptotic expansion of $F$.
We will apply an argument from asymptotic analysis, rewriting $F$ by means of Parseval formula for Mellin transforms, thus translating the problem into an Abelian theorem relying on properties of Mellin transforms which can be represented in terms of special functions.
The following paragraph collects all necessary preliminaries.

\subsection{Mellin Transforms and Parseval Formula}\label{subsec:abeliantheorem}

Let
\begin{equation*}
    J(\lambda)=\int_0^\infty h(\lambda t) f(t) d t,
\end{equation*}
where $h,f$ are locally integrable functions such that $J(\lambda)$ 
is absolutely convergent for large enough $\lambda>0$.
The Mellin transform of $f$ is defined by
\begin{equation*}
    M[f,z]=\int_0^\infty t^{z-1} f(t) dt
\end{equation*}
for $z\in \C$ for which the integral is absolutely convergent. When $f(t)$ (resp. $h(t)$) satisfies
\begin{equation*}
    f(t)\sim t^{a}, \quad t\to 0; 
        \quad
        f(t)\sim t^{b}, \quad t\to \infty,
\end{equation*}
with $a>b$, its Mellin transform 
is well-defined (the integral absolutely converges) on the vertical strip $-a<\re z<-b$ (its \emph{fundamental strip}).
We will make essential use of the Parseval formula for Mellin transforms:
\begin{equation}\label{eq:parseval}
    J(\lambda)=\frac1{2\pi i}\int_{r-i\infty}^{r+i\infty} \frac{M[h,z]M[f,1-z]}{\lambda^z} dz,
\end{equation}
which holds for any $r\in\R$ such that the line $r+i\R$ is entirely included in the intersection of the fundamental strips of the
involved Mellin transforms, provided that the integral on the right-hand side is absolutely convergent
(see \cite[Section 3.1.3]{Paris2001}).

If both $M[f,z]$ and $M[h,z]$ are well-defined on non-empty fundamental strips with non-empty intersection, and they can be 
analytically continued to the whole $\C$ as meromorphic functions, the difference of the right-hand side of \eqref{eq:parseval} for two different values of $r$ can be evaluated by regarding it as a contour integral (over the boundary of a strip) and applying residue calculus. 
More specifically, if $r>0$ lies in the fundamental strips of both $M[h,z]$ and $M[f,1-z]$, and 
$r'>r$ is such that $r'+i\R$ does not contain poles of the meromorphic functions $M[f,z]$ and $M[h,z]$,
\begin{equation}\label{eq:parsevaldiff}\begin{split}
    J(\lambda) &=\frac1{2\pi i}\int_{r-i\infty}^{r+i\infty} \frac{M[h,z]M[f,1-z]}{\lambda^z} dz\\
    & =\sum_{r<\re z< r'} \res\set{-\lambda^{-z}M[h,z]M[f,1-z]}+\frac1{2\pi i}\int_{r'-i\infty}^{r'+i\infty} \frac{M[h,z]M[f,1-z]}{\lambda^z} dz,
\end{split}\end{equation}
so that a suitable bound for the integral on the right-hand side implies an asymptotic expansion for $J(\lambda)$. 
This is a classical procedure in asymptotic analysis, we refer to the monograph \cite{Paris2001} for a detailed treatment.



In order to apply this argument in our specific case we will resort to representations of Mellin transforms in terms of special functions.
We will make extensive use of Euler's Gamma function $\Gamma(z)=M[e^{-t},z]$, and we recall its functional equation, the reflection and duplication formulae,
\begin{equation*}
    z\Gamma(z)=\Gamma(z+1),\quad 
    \pi=\Gamma(z)\Gamma(1-z)\sin(\pi z), \quad
    \Gamma(z)\Gamma(z+1/2)=\sqrt\pi 2^{1-2z}\Gamma(2z),
\end{equation*}
which we will often use (without further mention) to perform calculations and simplify formulae.
We recall the standard Beta integral: for $a>0$, $0<\re z<2a$,    \begin{equation}\label{eq:mellinbeta}
        \int_0^\infty \frac{t^{z-1}}{(1+t^2)^a}dt=
        \frac{\Gamma(z/2)\Gamma(a-z/2)}{2\Gamma(a)}
    \end{equation} 
(see for instance \cite[6.2 (30)]{Erdelyi1954}). 
We will also use the following non-trivial integral representation:

\begin{lemma}\label{lem:integral}
    For $0<s<d/2+1$, $0<\re w<2s$,
    \begin{equation}\label{eq:mellinf}
        \int_0^\infty t^{w-1} \int_0^\pi \frac{(\sin\theta)^d }{(1-2t \cos \theta+t^2)^s}d\theta dt
        =\frac{\sqrt\pi \gamhalf{d-2s+2}\gamhalf{d+1}}{2\Gamma(s)}
        \cdot \frac{\gamhalf{2s-w}\gamhalf{w}}{\gamhalf{d-w+2}\gamhalf{d-2s+2+w}}.
    \end{equation}
\end{lemma}

\begin{proof}
    \Cref{eq:mellinf} can be obtained by combining known integral representations of special functions. In particular, we recall that:
\begin{itemize}
    \item for $s>0,$ $\theta \in (-\pi,\pi)$ and $0<\re z<2s$,
    \begin{align}\label{eq:mellinlegendre}
        \int_0^\infty \frac{t^{z-1}}{(1+2t \cos\theta+t^2)^s}dt
        &=
        \frac{2^s \Gamma(s+1/2) \Gamma(z)\Gamma(2s-z)}{\sqrt 2 \Gamma(2s)} (\sin\theta)^{s-1/2} P^{1/2-s}_{z-s-1/2}(\cos \theta)\\ \nonumber
        &=
        \frac{\sqrt{2\pi} \Gamma(z)\Gamma(2s-z)}{2^s \Gamma(s)} (\sin\theta)^{s-1/2} P^{1/2-s}_{z-s-1/2}(\cos \theta),
    \end{align}
    where $P^\mu_\nu$ is the associated Legendre function of first kind, see \cite[14.3]{Olver2010};
    \item for $2\re \eta> |\re \mu|$,
    \begin{multline}\label{eq:powerlegendre}
        \int_{-1}^1 (1-x^2)^{\eta-1} P^\mu_\nu(x)dx
        =\frac{\pi 2^\mu \Gamma(\eta+\mu/2)\Gamma(\eta-\mu/2)}{\Gamma(\eta+\nu/2+1/2)\Gamma(\eta-\nu/2)\Gamma(1+\nu/2-\mu/2)\Gamma(1/2-\nu/2-\mu/2)}.
    \end{multline}
\end{itemize}
The formula \eqref{eq:powerlegendre} is found in \cite[7.132 (1)]{Gradshteyn2007}. The formula \eqref{eq:mellinlegendre} is reported in
\cite[6.2 (22)]{Erdelyi1954} and \cite[3.252 (10)]{Gradshteyn2007}, but those references contain typographical errors: an
\emph{errata corrige} including a derivation was published in \cite{Young1970} for an older edition of \cite{Gradshteyn2007}.
\end{proof}

\subsection{Asymptotic expansion of flux functions}\label{subsec:highlow}

\begin{proposition}\label{prop:highfreqasymp}
    Let $\alpha\in (0,1)$ and $s\in (0,d/2)$, $F$ be defined as in \eqref{eq:functions_Fn}. 
    There exist finite constants $K_{d,\alpha,s},C_{d,\alpha,s}>0$ depending on their subscripts such that
    \begin{equation}\label{eq:mainasymp}
        |F(\xi) +K_{d,\alpha,s} |\xi|^{2-2\alpha-2s}|\leq C_{d,\alpha,s} |\xi|^{-2s} \quad \forall\,\xi\in \R^d\setminus\{0\},
    \end{equation}
    where
    \begin{equation}\label{eq:formula_K_dalphas}
        K_{d,\alpha,s}= -
        \frac{(d+1) \Gamma(s+\alpha)\Gamma(-\alpha)\gamhalf{d-2s+2}}{2^{d/2}\pi^{1/2}\Gamma(s)\gamhalf{d+2\alpha+2}\gamhalf{d-2s+2-2\alpha}}>0.
    \end{equation}
\end{proposition}

\begin{proof}
For $\xi\neq0$, $\alpha\in (0,1)$ and $s\in (0,d/2)$,
we may rewrite the integral \eqref{eq:functions_Fn} that defines $F$ as
\begin{align}\label{eq:Fdecomp}
    F(\xi)&= (2\pi)^{-d/2}\int_{\R^d} \frac{|P_{\xi-\eta}^\perp \xi|^2}{\brak{ \xi-\eta}^{d+2\alpha}} 
    \pa{\frac{1}{|\eta|^{2s}}-\frac{1}{|\xi|^{2s}}} \dd \eta\\ \nonumber
    &= (2\pi)^{-d/2} \bigg[ \int_{\R^d} \frac{|P_{\eta'}^\perp \xi|^2}{\brak{\eta'}^{d+2\alpha}} \frac{1}{|\xi-\eta'|^{2s}} \dd \eta' 
    - \frac{1}{|\xi|^{2s}} \int_{\R^d} \frac{|P_{\eta'}^\perp \xi|^2}{\brak{\eta'}^{d+2\alpha}} \dd \eta'\bigg]=: (2\pi)^{-d/2} [I(\xi)-G(\xi)].
\end{align}
We will analyze the two summands separately.
Using the symmetries of the projection operator $P^\perp$ and passing to polar coordinates, we can compute:
\begin{align*}
    G(\xi)
    & = \frac{1}{|\xi|^{2s}} \int_{\R^d} \frac{|P_{\eta}^\perp(|\xi| e_1)|^2}{\brak{\eta}^{d+2\alpha}} \dd \eta
    = |\xi|^{2-2s} \int_{\R^d} \frac{1-\eta_1^2/|\eta^2|}{\brak{\eta}^{d+2\alpha}} d\eta\\
    & = |\xi|^{2-2s} \omega_{d-2} \int_0^{+\infty} \frac{r^{d-1}dr}{(1+r^2)^{d/2+\alpha}} \int_0^\pi \sin^d\theta d\theta
    = \omega_{d-2} \frac{\Gamma\left(d/2\right) \Gamma\left(\alpha\right)}{2 \Gamma\left(d/2+\alpha\right)}\cdot \frac{\sqrt\pi \gamhalf{d+1}}{\gamhalf{d+2}} |\xi|^{2-2s},
\end{align*}
the last step following from \eqref{eq:mellinbeta} and another standard Beta integral:
for $\gamma,\eta>-1$,
\begin{equation}\label{eq:mellinbetatrig}
        \int_0^\pi (\sin \theta)^\gamma (\cos \theta)^\eta d \theta=
        \frac{\gamhalf{\gamma+1}\gamhalf{\eta+1}}{\gamhalf{\gamma+\eta+2}}
\end{equation} 
(see \cite[3.621 (5)]{Gradshteyn2007}).
We are thus left to study the term $I(\xi)$, which by using radial symmetries and polar coordinates can be written as
\begin{align*}
    I(\xi)
    & = \int_{\R^d} \frac{|P_{\eta}^\perp (|\xi|e_1)|^2}{\brak{\eta}^{d+2\alpha}} \frac{1}{||\xi|e_1-\eta|^{2s}} \dd \eta
      = |\xi|^{d+2-2s} \int_{\R^d} \frac{|P_{\eta}^\perp (e_1)|^2}{\brak{|\xi|\eta}^{d+2\alpha}} \frac{1}{|e_1-\eta|^{2s}} \dd \eta\\
    & = |\xi|^{d+2-2s} \int_{\R^d} \frac{1-\eta_1^2/|\eta|^2}{(1 + |\xi|^2 |\eta|^2)^{d/2+\alpha}}\, \frac{1}{|1-2\eta_1 + |\eta|^2|^s} \dd \eta\\
    & = |\xi|^{d+2-2s} \int_0^{+\infty} \frac{r^{d-1}}{(1+|\xi|^2 r^2)^{d/2+\alpha}} \int_{S^{d-1}} \frac{\sin^2 \theta_1}{|1-2r\cos\theta_1+r^2|^s}\dd \sigma_{d-1} \dd r\\
    & = \omega_{d-2} |\xi|^{d+2-2s} \int_0^{+\infty} \frac{r^{d-1}}{(1+|\xi|^2 r^2)^{d/2+\alpha}} \int_0^\pi \frac{(\sin\theta)^d }{|1-2r \cos \theta+r^2|^{s}} d\theta \dd r.
\end{align*}
Setting $\lambda=|\xi|$, we are therefore reduced to study the integral
\begin{equation*}
    J(\lambda)=\int_0^\infty h(\lambda r) f(r) dr,\quad 
    f(r) =r^{d-1} \int_0^\pi \frac{\sin^d\theta d\theta}{|1-2r \cos \theta+r^2|^s},
    \quad 
    h(r)=\frac{1}{\left(1+r^2\right)^{d / 2+\alpha}},
\end{equation*}
to which we apply the strategy outlined in the previous paragraph.

By \eqref{eq:mellinbeta}, the Mellin transform of $h$ is analytically continued to the meromorphic function
\begin{equation}\label{eq:mellinh}
     M[h, z]=\frac{\Gamma\left(z/2\right) \Gamma\left(d/2+\alpha-z/2\right)}{2 \Gamma\left(d/2+\alpha\right)},
\end{equation}
having simple poles at points $-2\N \cup (d+2\alpha+2\N)$, with fundamental strip $0<\re z< d+2\alpha$.
The relevant pole in our computation is $z=d+2\alpha$, for which:
\begin{equation*}
    \res_{z=d+2\alpha}\bra{M[h,z]}=-1.
\end{equation*}
By \eqref{eq:mellinf}, the Mellin transform of $f$ is analytically continued to the meromorphic function
\begin{equation*}
    M[f,1-z]
    =\int_0^\infty r^{d-1-z} \int_0^\pi \frac{\sin^d\theta d\theta}{|1-2r \cos \theta+r^2|^s}
    =\frac{\sqrt\pi \gamhalf{d-2s+2}\gamhalf{d+1}}{2\Gamma(s)}
        \cdot \frac{\gamhalf{2s-d+z}\gamhalf{d-z}}{\gamhalf{z+2}\gamhalf{2d-2s+2-z}},    
\end{equation*}
having simple poles at points $(d+2\N)\cup (d-2s-2\N)$, with fundamental strip $d-2s< \re z <d$ (a subset of that of $M[h,z]$).
The relevant poles in our computation are $z=d$ and $z=d+2$, for which:
\begin{equation*}
    \res_{z=d}\bra{M[f,1-z]}=
        -\frac{\sqrt\pi \gamhalf{d+1}}{\gamhalf{d+2}},\qquad
    \res_{z=d+2}\bra{M[f,1-z]}=
        \frac{\sqrt\pi (d-2s+2)(s+1) \gamhalf{d+1}}{4\gamhalf{d+4}}.
\end{equation*}

We can now apply \eqref{eq:parsevaldiff} in order to obtain an asymptotic expansion for $J(\lambda)$. 
We choose $r\in (d-2s,d)$, the latter being the intersection of fundamental strips,
and as the right extremum we choose $r'\in (d+2, d+2\alpha+2)$. 
Since $\alpha\in (0,1)$, the analytic continuation of $M[h,z]$ and $M[f,1-z]$ have disjoint sets of singularities in $(r,r')$.
In the strip $d-2s<\re z < d+2\alpha+2$, $M[h,z]$ has a single simple pole in $z=d+2\alpha$ and
\begin{align*}
    \res_{z=d+2\alpha}\bra{-\lambda^{-z}M[h,z]M[f,1-z]}
    & =\frac{1}{\lambda^{d+2\alpha}}M[f,1-d-2\alpha]\\
    & =\frac{1}{\lambda^{d+2\alpha}} 
    \frac{\sqrt\pi \Gamma(s+\alpha)\Gamma(-\alpha)\gamhalf{d-2s+2}\gamhalf{d+1}}{2\Gamma(s)\gamhalf{d+2\alpha+2}\gamhalf{d-2s+2-2\alpha}}.
\end{align*}
In the chosen strip, $M[f,1-z]$ has a simple pole in $z=d$ with
\begin{equation*}
    \res_{z=d}\bra{-\lambda^{-z}M[h,z]M[f,1-z]}
    = \frac{1}{\lambda^d} \cdot
    \frac{\Gamma\left(d/2\right) \Gamma\left(\alpha\right)}{2 \Gamma\left(d/2+\alpha\right)}\cdot \frac{\sqrt\pi \gamhalf{d+1}}{\gamhalf{d+2}},
\end{equation*}
and another simple pole in $z=d+2$, with
\begin{equation*}
    \res_{z=d+2}\bra{-\lambda^{-z}M[h,z]M[f,1-z]}
    = -\frac{1}{\lambda^{d+2}} \cdot \frac{\gamhalf{d+2} \Gamma(\alpha-1)}{2 \Gamma\left(d/2+\alpha\right)} \cdot \frac{\sqrt\pi (d-2s+2)(s+1) \gamhalf{d+1}}{4\gamhalf{d+4}}
\end{equation*}
(which is a positive constant).
We are left to bound the contribution of the integral on the right-hand side of \eqref{eq:parsevaldiff}, in order to do so we recall the standard asymptotic relation for the Gamma function
\begin{equation}\label{eq:gammaasymp}
    \lim_{|y|\to \infty} \abs{\Gamma(x+iy)} e^{\pi |y|/2} |y|^{1/2-x}=\sqrt{2\pi}
\end{equation}
(the asymptotic relation being uniform for bounded $x\in\R$, see \cite[5.11.9]{Olver2010}).
The meromorphic functions $M[h,z]$, $M[f,1-z]$ do not have poles on $r'+i\R$, $r'\in (d+2,d+2\alpha+2)$, and therefore they are bounded on bounded intervals of that line,
so by \eqref{eq:gammaasymp} we have
\begin{equation*}
    |M[h,r'+iy]|\lesssim_{d,\alpha,s}  (1\vee |y|)^{d/2+\alpha-1}e^{-\pi |y|/2},\quad
    |M[f,1-r'-iy]|\lesssim_{d,\alpha,s} (1\vee |y|)^{2s-2-d},
\end{equation*}
uniformly in $y\in\R$.
As a consequence, we obtain the following bound on the integral on the right-hand side of \eqref{eq:parsevaldiff}:
\begin{equation*}
    \abs{\frac1{2\pi i}\int_{r'-i\infty}^{r'+i\infty} \frac{M[h,z]M[f,1-z]}{\lambda^z} dz}
    =O_{d,\alpha,s}(\lambda^{-r'}),
\end{equation*}
in particular the integral absolutely converges.

We have thus proved that
\begin{equation*}
    J(\lambda)= \frac{1}{\lambda^d} \cdot
    \frac{\Gamma\left(d/2\right) \Gamma\left(\alpha\right)}{2 \Gamma\left(d/2+\alpha\right)}\cdot \frac{\sqrt\pi \gamhalf{d+1}}{\gamhalf{d+2}}
    +\frac{1}{\lambda^{d+2\alpha}} 
    \frac{\sqrt\pi \Gamma(s+\alpha)\Gamma(-\alpha)\gamhalf{d-2s+2}\gamhalf{d+1}}{2\Gamma(s)\gamhalf{d+2\alpha+2}\gamhalf{d-2s+2-2\alpha}}
    +O_{d,\alpha,s}(\lambda^{-d-2}).
\end{equation*}
Recalling now $I(\xi)=\omega_{d-2} |\xi|^{d+2-2s} J(|\xi|)$, we observe that the leading order term cancels exactly with $G(\xi)$,
and the remaining terms of the asymptotic relation we established are those appearing in the thesis.
\end{proof}

\section{Rigorous statements and proofs of the main results}\label{subsec:mainproofs}

We are now ready to present the rigorous versions of our main  results, that were informally stated in the introduction as \cref{thm:main_informal}.

\begin{theorem}\label{thm:main1}
	Let $d\geq 2$, $\alpha\in (0,1)$, $s\in (0,d/2)$ and $W$ be given by \eqref{eq:kraichnan_noise}; let $\rho_0\in L^1\cap L^2$ and let $\rho$ be the unique solution to \eqref{eq:transportnoise}, given by \cref{thm:wellposedness_kraichnan}.
	Then in the sense of distributions, it holds
	\begin{equation}\label{eq:main_ineq1}
	 	\frac{\dd}{\dd t} \EE[\|\rho_t \|_{\dot H^{-s}}^2] + K_{d,\alpha,s}\, \EE[\|\rho_t \|_{\dot H^{-s-\alpha+1}}^2] \leq C_{d,\alpha,s}\, \EE[\|\rho_t \|_{\dot H^{-s}}^2].
    \end{equation}
    where $K_{d,\alpha,s}$, $C_{d,\alpha,s}>0$ are the same constants appearing in \cref{prop:highfreqasymp}.
	As a consequence, for any $\rho_0\in L^1\cap L^2$, $T\in (0,+\infty)$ and $s$ as above, we have
	\begin{equation}\label{eq:main_ineq2}
		\sup_{t\in [0,T]} \EE[\|\rho_t \|_{\dot H^{-s}}^2] + K_{d,\alpha,s} \int_0^T \EE[\|\rho_t \|_{\dot H^{-s-\alpha+1}}^2] \dd t \leq 2 e^{C_{d,\alpha,s} T} \, \EE[\|\rho_0 \|_{\dot H^{-s}}^2].
	\end{equation}
\end{theorem}

\begin{proof}
%
%
Once we prove \eqref{eq:main_ineq1}, estimate \eqref{eq:main_ineq2} is an immediate consequence of Gr\"onwall's lemma.

In the case $s\in (1-\alpha,d/2)$, estimate \eqref{eq:main_ineq1} follows from combining \cref{prop:identity_neg_sobolev} and \cref{prop:highfreqasymp}.
Instead, in the case $s\in (0,1-\alpha]$, we first consider the viscous solution $\rho^\nu$ to \eqref{eq:transportnoise_viscous} and apply \cref{prop:identity_neg_sobolev_general}-\cref{prop:highfreqasymp} to find
\begin{equation}\label{eq:main_proof_eq1}
    \frac{\dd}{\dd t} \EE[\|\rho^\nu_t \|_{\dot H^{-s}}^2] + K_{d,\alpha,s} \EE[\|\rho^\nu_t \|_{\dot H^{-s-\alpha+1}}^2] \leq -2\nu \EE[\|\rho^\nu_t \|_{\dot H^{-s+1}}^2] + C_{d,\alpha,s} \EE[\|\rho^\nu_t \|_{\dot H^{-s}}^2];
\end{equation}
notice that, again by Gr\"onwall, this implies the uniform-in-$\nu$ estimate
\begin{equation}\label{eq:main_proof_eq2}
    \sup_{t\in [0,T]} \EE[\|\rho^\nu_t \|_{\dot H^{-s}}^2] + K_{d,\alpha,s} \int_0^T \EE[\|\rho^\nu_t \|_{\dot H^{-s-\alpha+1}}^2] \dd t \lesssim \| \rho_0\|_{\dot H^{-s}}^2.
\end{equation}
 We now want to pass to the limit in \eqref{eq:main_proof_eq1} as $\nu\to 0$, aided by the a priori estimates \eqref{eq:bounds_viscous_kraichnan} and  \cref{prop:vanishing_viscosity_approximation}.
Indeed, notice that by interpolation
\begin{align*}
    2\nu \int_0^T \EE[\|\rho^\nu_t \|_{\dot H^{-s+1}}^2] \dd t
    \lesssim \bigg( 2\nu \int_0^T \EE[\|\rho^\nu_t \|_{\dot H^1}^2] \dd t\bigg)^{1-s} \bigg( 2\nu \int_0^T \EE[\|\rho^\nu_t \|_{L^2}^2] \dd t\bigg)^s
    \lesssim (2\nu)^s T^s \|\rho_0\|_{L^2}^2 \to 0
\end{align*}
as $\nu\to 0^+$, so that the r.h.s. of \eqref{eq:main_proof_eq1} converges strongly in $L^2([0,T])$ to $C_{d,\alpha,s} \EE[\|\rho_t \|_{\dot H^{-s}}^2]$ as $\nu\to 0^+$.
On the other hand, since $\rho^\nu\to \rho$ strongly in $C_t \dot H^{-s}$ and by \eqref{eq:main_proof_eq2} the sequence is bounded in $L^2_{\omega,t} \dot H^{-s+1-\alpha}$, $\rho^\nu \in L^2_{\omega,t} \dot H^{-s+1-\alpha}$ as well. Upon replacing $s$ with $\tilde s=s-\eps>0$ for $\eps$ small enough, we can interpolate between the convergence in $L^2_{\omega,t} \dot H^{-1/2}$ (coming from \cref{prop:vanishing_viscosity_approximation}) and the uniform bound in $L^2_{\omega,t} \dot H^{-\tilde s -\alpha+1}$ to conclude that
\begin{equation}\label{eq:main_proof_eq3}
    \lim_{\nu\to 0^+} \int_0^T \EE\big[\,\| \rho^\nu_t-\rho_t\|_{\dot H^{-s+1-\alpha}}^2\big] = 0 \quad \forall\, s\in (0,d/2).
\end{equation}
All in all, this implies that all terms appearing in \eqref{eq:main_proof_eq1} converge in distribution to their inviscid counterparts and so that inequality \eqref{eq:main_ineq1} holds for $s\in (0,1-\alpha]$ as well.
\end{proof}

\begin{remark}
    Upon taking $s=1-\alpha$, estimate \eqref{eq:main_proof_eq3} in particular implies that
    \begin{equation*}
    \lim_{\nu\to 0^+} \int_0^T \EE\big[\,\| \rho^\nu_t-\rho_t\|_{L^2}^2\big] = 0,
\end{equation*}
    namely viscous solutions $\rho^\nu$ converge strongly to $\rho$ in $L^2_{\omega,t,x}$. This is not trivial, in light of the fact that $\rho$ is not renormalized and anomalous dissipation of energy takes place in the Kraichnan model.
\end{remark}

We denote by $S:L^1\cap L^2 \to L^\infty_{\omega,t}(L^1\cap L^2)$ the solution map to SPDE \eqref{eq:transportnoise} coming from \cref{thm:wellposedness_kraichnan}, that sends every initial condition $\rho_0\in L^1\cap L^2$ to the corresponding unique solution $\rho=S\rho_0$. As equation \eqref{eq:transportnoise} is linear, so is the map $S$.

\begin{theorem}\label{thm:main2}
	Let $d$, $\alpha$, $s$ and $W$ be as in \cref{thm:main1}.
	Then the solution map $S$ extends uniquely to a linear map from $\dot H^{-s}$ to $L^\infty_t L^2_\omega \dot H^{-s}$; for every $\rho_0\in \dot{H}^{-s}$, the process $\rho = S\rho_0$ is a $\dot{H}^{-s}$-solution to \eqref{eq:transportnoise} (in the sense of \cref{defn:Kraichnan}) and still satisfies \eqref{eq:main_ineq2} (for such fixed value of $s$).

    The process $\rho$ has the properties that, for every $\eps>0$ and every $0<\delta<T$, 
	\begin{equation}\label{eq:arbitrary_gain}
		\int_\delta^T \EE[\|\rho_t\|_{\dot{H}^{1-\alpha-\eps}}^2] \dd t <\infty
	\end{equation}
    and $\PP$-a.s. it holds that
    \begin{equation}\label{eq:arbitrary_gain2}
		\sup_{t\geq \delta} \|\rho_t\|_{L^2} \leq \|\rho_\delta\|_{L^2}<\infty.
	\end{equation}
\end{theorem}

\begin{proof}
The unique extension of the map $S$ comes from its linearity and estimate \eqref{eq:main_ineq2}. The fact that, for $\rho_0\in \dot{H}^{-s}$, $S\rho_0$ is a solution to \eqref{eq:transportnoise} follows by a standard approximation procedure: we take $\rho_0^n\in L^1\cap L^2$ converging to $\rho_0$ in $\dot{H}^{-s}$ and we pass to the limit, with respect to the $L^2_\omega C_t H^{-s-2}$ topology, in the equation \eqref{eq:def_sol} for $S\rho^n_0$; the convergence of the stochastic integral in \eqref{eq:def_sol} is a consequence of \cref{lem:stoch_int_distributions}.

Finally \eqref{eq:arbitrary_gain} follows by bootstrapping the regularity gain coming from \eqref{eq:main_ineq2}. Precisely, if we start with $\rho_0 \in \dot{H}^{-s}$, by \eqref{eq:main_ineq2} there exists $t_1$ arbitrarily small such that $\EE[\|\rho_{t_1}\|_{\dot{H}^{-s_1}}^2]<\infty$, where $s_1=s+\alpha-1$. Restarting from $t_1$, by \eqref{eq:main_ineq2} we have
\begin{align}\label{eq:iteration}
    \EE\left[\int_{t_1}^T \|\rho_r\|_{\dot{H}^{-s_2}}^2 \mid \mathcal{F}_{t_1} \dd t \right]= \EE\left[\EE\left[\int_{t_1}^T \|\rho_r\|_{\dot{H}^{-s_2}}^2 \dd t \mid \mathcal{F}_{t_1} \right] \right] \le 2e^{C_{d,\alpha,s_1}T} \EE[\|\rho_{t_1}\|_{\dot{H}^{-s_1}}^2]<\infty,
\end{align}
so there exists $t_2$ arbitrarily small such that $\EE[\|\rho_{t_2}\|_{\dot{H}^{-s_2}}^2]<\infty$, where $s_2=s+2(\alpha-1)$. Iterating this procedure a finite number of times (until $-s_n$ becomes non-negative), and interpolating with $\EE[\|\rho_t\|_{\dot{H}^{-s}}^2]<\infty$ for every $t$, we obtain that, for every $\eps>0$, there exists $t_n$ arbitrarily small (in particular, $t_n\leq \delta$) such that $\EE[\|\rho_{t_n}\|_{L^2}^2]<\infty$ and $\EE[\|\rho_{t_n}\|_{\dot{H}^{-\eps}}^2]<\infty$; applying once more \eqref{eq:iteration} with $s_1$ replaced by $\eps$ and $t_1$ replaced by $t_n$, we reach \eqref{eq:arbitrary_gain}.

By the same argument (and possibly interpolating between $\dot H^{-\alpha}$ and $\dot H^\beta$), we can find $\tau\leq \delta$ small enough such that $\EE[\|\rho_\tau\|_{L^2}^2]<\infty$. In this case, restarting from $\tau$, estimate \eqref{eq:Lp_bound} guarantees that $\rho_t\in L^2$ for all $t\geq \tau$; restarting from $\delta\geq \tau$ then yields \eqref{eq:arbitrary_gain2}.
\end{proof}


\section{On the statistically self-similar rough Kraichnan case}\label{sec:self-similar}

As already mentioned in the introduction, in the Physics literature authors have often restricted themselves to the \emph{statistically self-similar} Kraichnan noise. This formally amounts to replacing the noise $W$ in the SPDE \eqref{eq:transportnoise} with $\tilde W$, whose covariance $\tilde Q$ has Fourier transform given by (compare with \eqref{eq:kraichnan_noise})
\begin{align*}
    \widehat{\tilde Q}(\xi) = \frac{1}{|\xi|^{d+2\alpha}} P^\perp_\xi.
\end{align*}
Although idealized, self-similar noise has the advantage of often allowing to perform explicit computations by enforcing a self-similar Ansatz, while leaving the small-scale structure of the advected scalar roughly unaffected.

More rigorously, in order to study solutions $\tilde \rho$ to the SPDE \eqref{eq:transportnoise} associated to $\tilde W$, one should perform an approximation procedure, by first considering solutions $\rho^m$ to the SPDE
\begin{equation}\label{eq:transport_equation_m}
    \dd \rho^m +\circ \dd W^m\cdot\nabla \rho^m = 0, \quad \rho^m\vert_{t=0}=\rho_0
\end{equation}
with noises $W^m$ whose covariance $Q^m$ is determined by
\begin{align*}
    \widehat{Q}^m(\xi) = \frac{1}{(m^2 + |\xi|^2)^{d/2+\alpha}} P^\perp_\xi,
\end{align*}
and then passing to the limit as $m\to 0^+$.
This limit however is far from trivial, as the SPDEs \eqref{eq:transport_equation_m} do not make sense after the limit procedure: when writing them in It\^o form (cf. \eqref{eq:transportnoise}), the It\^o--Stratonovich corrector ${\rm Tr}\,(Q^m(0))\Delta \rho^m$ does not have a limit in the sense of distributions, since ${\rm Tr}\,(Q^m(0))\to +\infty$ as $m\to 0^+$. In other contexts, this can lead to triviality results, i.e. the limit solutions $\rho$ being spatially constant, cf. \cite[Thm. 4.1]{FGL2021}.

Nevertheless, many quantities of interest (like the evolution of $\dot H^{-s}$-norms, Lyapunov exponents, or the generator of the associated two-point motion \cite{CaFaGa2008,zelati2023statistically}) are not governed by $Q^m(0)$, but rather by the functions $x\mapsto Q^m(0)-Q^m(x)$, which admit a pointwise limit:
\begin{equation}\label{eq:covariance_selfsimilar}
    \lim_{m\to 0} (Q^m(0)-Q^m(z)) = D_{\alpha,d}\, |z|^{2\alpha} \Big[ I_d + \frac{2\alpha}{d-1} P^\perp_z\Big]\quad \forall\, z\in\R^d
\end{equation}
for a suitable finite constant $D_{\alpha,d}$.

In light of the above facts, it is of interest to understand what happens to the evolution of Sobolev norms in the self-similar case, i.e. in the limit $m\to 0^+$.
Since we do not have a well-defined limit for the SPDEs \eqref{eq:transport_equation_m}, the content of this section is partially heuristical, thus we will refrain from giving fully rigorous statements.
We will simply assume that, for fixed deterministic $\rho_0\in L^1\cap L^2$, the solutions $\rho^m$ to \eqref{eq:transport_equation_m} converge in suitable topologies to some limit $\tilde\rho$, which we will somewhat regard as a solution to the transport SPDE driven by the statistically self-similar noise $\tilde W$ of parameter $\alpha\in (0,1)$. The results obtained, albeit formal, are still informative when regarded as ideal limits of more practical cases.

In analogy to \eqref{eq:fourier-balance}-\eqref{eq:functions_Fn}, going through the same proof as in \cref{prop:identity_neg_sobolev}, one finds
\begin{equation}\label{eq:fourier-balance_m}
    \frac{\dd}{\dd t} \EE\big[\| \rho^m_t\|_{\dot H^{-s}}^2\big]
    = (2\pi)^{-d/2} \int_{\R^d} F^m(\xi)\, \EE\big[|\widehat{\rho}^m_t(\xi)|^2\big] \dd \xi
\end{equation}
for the flux functions $F^m$ given by
\begin{equation}\label{eq:functions_Fn_m}
        F^m(\xi):= \int_{\R^d} \frac{1}{(m^2 + |\eta|^2)^{d/2+\alpha}} |P_{\eta}^\perp \xi|^2
    \Big(\frac{1}{|\eta-\xi|^{2s}}-\frac{1}{|\xi|^{2s}}\Big) \dd \eta.
\end{equation}
A simple rescaling $\eta=m\eta'$ reveals the identity
\begin{align*}
    F^m(\xi) = m^{2-2s-2\alpha} F\Big(\frac{\xi}{m}\Big);
\end{align*}
once we plug it in \eqref{eq:mainasymp}, we obtain 
\begin{equation}\label{eq:functions_Fn_m_limit}
    |F^m(\xi) + K_{d,\alpha,s} |\xi|^{2-2\alpha-2s} | \leq C_{d,\alpha,s} m^{2-2\alpha} |\xi|^{-2s} \quad \forall \xi\in\R^d\setminus\{0\},\, m>0.
\end{equation}
In particular, $F^m(\xi)$ converge to $-K_{d,\alpha,s} |\xi|^{2-2\alpha-2s}$ as $m\to 0^+$, uniformly outside neighborhoods of $0$.
As a consequence, passing to the limit as $m\to 0$ in \eqref{eq:fourier-balance_m}, we expect the evolution of negative Sobolev norms to solutions $\tilde \rho$ transported by the $\alpha$-self-similar Kraichnan noise $\tilde W$ to be given exactly by
\begin{equation}\label{eq:energy-balance-selfsimilar}
    \frac{\dd}{\dd t} \EE\big[\| \tilde \rho_t\|_{\dot H^{-s}}^2\big]
    = - K_{d,\alpha,s} \EE\big[\| \tilde \rho_t\|_{\dot H^{1-\alpha-s}}^2\big]
\end{equation}
for any $s\in (0,d/2)$.

\begin{remark}
    A posteriori, the limit $m\to 0^+$ provides another way to compute the leading constant $K_{d,\alpha,s}$ from \cref{prop:highfreqasymp}: by \eqref{eq:functions_Fn_m}-\eqref{eq:functions_Fn_m_limit}, it holds
    \begin{align*}
        K_{d,\alpha,s} |\xi|^{2-2\alpha-2s}
        = (2\pi)^{-d/2} \int_{\R^d} \frac{1}{|\eta|^{d+2\alpha}} |P_{\eta}^\perp \xi|^2 \Big(\frac{1}{|\eta-\xi|^{2s}}-\frac{1}{|\xi|^{2s}}\Big) \dd \eta,
    \end{align*}
    so that after rescaling and rotating one finds
    \begin{align*}
        K_{d,\alpha,s} = (2\pi)^{-d/2} \int_{\R^d} \frac{1}{|\eta|^{d+2\alpha}} |P_{\eta}^\perp e_1|^2 \Big(1 -\frac{1}{|\eta-e_1|^{2s}}\Big) \dd \eta.
    \end{align*}
\end{remark}

\begin{remark}[Energy decay]
    Upon integrating \eqref{eq:energy-balance-selfsimilar} on $(0,+\infty)$, for any deterministic $\rho_0\in L^1\cap L^2$ one finds
    \begin{equation}\label{eq:anomalous_dissipation_selfsimilar}
        \int_0^{+\infty} \EE[\| \tilde \rho_t\|_{L^2}^2] \dd t = \frac{1}{\tilde K_{d,\alpha}} \| \rho_0\|_{\dot H^{\alpha-1}}^2.
    \end{equation}
    Eq. \eqref{eq:anomalous_dissipation_selfsimilar} reveals in a more quantitative way the anomalous dissipation property of rough Kraichnan noise. In the statistically self-similar case, not only $\| \tilde\rho_t\|_{L^2}^2$ is not preserved by the dynamics, but it also has to decay sufficiently fast so as to make the integral in \eqref{eq:anomalous_dissipation_selfsimilar} finite. This is consistent with \cite{EyiXin2000}, predicting an asymptotic decay for $\EE[\| \tilde\rho\|_{L^2}^2]$ of order at least $t^{-1-\frac{d}{2-2\alpha}}$.
\end{remark}

\begin{remark}\label{rem:selfsimilar_forcing}
    Similarly to \cite{zelati2023statistically}, one could consider the SPDE with diffusion and an additive Brownian forcing $B_t=\sum_k q_k B^k_t$ independent of $\tilde W$, formally given by
    \begin{equation}\label{eq:selfimilar_SPDE_forcing}
        \dd \tilde\rho_t + \circ \dd \tilde W_t\cdot \nabla \tilde\rho_t = \nu \Delta \tilde\rho_t \dd t +  \dd B_t.
    \end{equation}
    Going through the same computations yielding \eqref{eq:energy-balance-selfsimilar}, one then finds
    \begin{equation}\label{eq:balance_selfsimilar_forcing}
        \frac{\dd}{\dd t} \| \tilde\rho_t\|_{\dot H^{-s}}^2 = - K_{d,\alpha,s} \EE [\| \tilde\rho_t\|_{\dot H^{1-s-\alpha}}^2] - 2\nu\, \EE [\| \tilde\rho_t\|_{\dot H^{1-s}}^2] + \sum_k \| q_k\|_{\dot H^{-s}}^2 
    \end{equation}
    which is in line with the case $\alpha=1$ treated in \cite{zelati2023statistically}.
    The Laplacian $\Delta$ may also be replaced by other operators (e.g. fractional Laplacian), yielding similar relations.    
\end{remark}

\begin{remark}[Invariant measures]\label{rem:invariant_measures}
    In the setting of \cref{rem:selfsimilar_forcing}, relation \eqref{eq:balance_selfsimilar_forcing} and the Krylov-Bogoliubov suggest the existence of an invariant measure $\mu^\nu$ for \eqref{eq:selfimilar_SPDE_forcing} such that
    \begin{equation*}
    K_{d,\alpha,s} \EE_{\mu^\nu} [\| \tilde\rho_t\|_{\dot H^{1-s-\alpha}}^2] + 2\nu\, \EE_{\mu^\kappa} [\| \tilde\rho_t\|_{\dot H^{1-s}}^2] = C_s \quad \forall\, s\in (0,d/2)
    \end{equation*}
    Sending $\nu\to 0^+$, one would expect the existence of an invariant measure $\mu$ for the inviscid, self-similar forced Kraichnan model such that
    \begin{equation*}
    K_{d,\alpha,s}\, \EE_{\mu} [\| \tilde\rho_t\|_{\dot H^{1-s-\alpha}}^2] = C_s \quad \forall\, s\in (0,d/2).
    \end{equation*}
    In particular, $\mu$ is supported on $H^{1-s-}$, which matches Kraichnan's original prediction \cite{Kraichnan1968} on the second order structure function at equilibrium scaling like $S_2(r)\sim r^{2-2\alpha}$ (see also \cite[eq. (296)]{MajKra1999}).

    Let us stress however that all of the above is still mathematically very formal.
    In the Physics literature, the existence of a unique statistically stationary state for \eqref{eq:selfimilar_SPDE_forcing}, which is approached at long times by all solutions with sufficient spatial decay, is often taken for granted, cf. \cite[Sec. 4.3.1.2]{MajKra1999}.
    Recently, rigorous results for the invariant measure of the (non-self-similar) Kraichnan model in the torus $\T^d$ have been established in \cite{Rowan2023}.
\end{remark}

\appendix

\section{Formal derivation in the self-similar case}\label{app:formal_derivation}


The goal of this appendix is to present a much shorter, very formal, derivation of the evolution of negative Sobolev norms in Kraichnan.
This is accomplished by running computations at the Eulerian level, in the statistically self-similar case, in which one obtains
\begin{equation}\label{eq:app_exact_balance}
    \frac{\dd}{\dd t} \EE[\|\rho_t \|_{\dot H^{-s}}^2] + K_{d,\alpha,s}\, \EE[\|\rho_t \|_{\dot H^{-s-\alpha+1}}^2] = 0 
\end{equation}
for the same constant $K_{d,\alpha,s}$ appearing in \eqref{eq:main_ineq1}. Let us stress that, for statistically self-similar $W$, the SPDE \eqref{eq:intro_kraichnan} is not even well-defined; the computations below should be regarded as a tool for driving the intuition, momentarily disregarding the more technical and rigorous mathematical details.

Recall the following general fact: given a smooth convolution kernel $G$, the evolution of the associated quadratic quantity along the inviscid Kraichnan SPDE \eqref{eq:intro_kraichnan} in expectation is given by
\begin{equation}\label{eq:app_kernel_balance}\begin{split}
    \frac{\dd}{\dd t} \mathbb{E} [\langle G \ast \rho_t, \rho_t \rangle]
    & =\mathbb{E} \int_{\mathbb{R}^d \times \mathbb{R}^d} {\rm Tr} ((Q (0) - Q (x  - y)) D^2 G (x - y)) \rho_t (x) \rho_t (y) \dd x \dd y\\
    & =\mathbb{E} [\langle \mathcal{H} \ast \rho_t, \rho_t \rangle],
\end{split}\end{equation}
where $Q$ is the covariance function of $W$ determined by \eqref{eq:covariance_W} and the new kernel $\mathcal{H}$ is defined by
\begin{align*}
    \mathcal{H} (z) = (Q (0) - Q (z)) : D^2 G (z),
\end{align*}
where $A:B$ denotes the Frobenius inner product for matrices, i.e. $A:B={\rm Tr} (A^T B) = \sum_{i,j} A_{ij} B_{ij}$. For alternative rigorous justifications of \eqref{eq:app_kernel_balance}, see \cite[Lem. 4.1]{coghi2023existence}, \cite[Lem. 1]{zelati2023statistically} or \cite[Lem. 4.4]{bagnara2024regularization}.
Recall that in the statistically self-similar case, up to the multiplicative constant $D_{\alpha,d}$ (set equal to $1$ here for simplicity), by \eqref{eq:covariance_selfsimilar} we have
\begin{align*}
    Q (0) - Q (z) = | z |^{2 \alpha} \left( P^{\|}_z + \left( 1 + \frac{2
   \alpha}{d - 1} \right) P^{\perp}_z \right).
\end{align*}
For $s\in (0,d/2)$, let $G^s (z) := |z|^{2s-d}$, which is the Riesz potential associated to $(-\Delta)^{-s}$ up to a multiplicative constant; it holds
\begin{align*}
    D^2 G^s (z) = (2s-d) |z|^{2s-d-2}\, \big[ (2s-d-1) P^{\|}_z +
   P^{\perp}_z\big].
\end{align*}
Using the relations
\begin{align*}
    P^{\|}_z : P^{\|}_z = 1, \quad
    P^{\perp}_z : P^{\perp}_z = d - 1, \quad
    P^{\|}_z : P^{\perp}_z = 0,
\end{align*}
we find the new kernel $\mathcal{H}^s$ associated to $G^s$ to be given by
\begin{equation}\label{eq:formula.H}
    \mathcal{H}^s(z) = 2 (2s-d) (s+\alpha-1) |z|^{2s+2\alpha-2-d} 
\end{equation}
Balance \eqref{eq:app_kernel_balance} rigorously holds for regular $G$ and non-self-similar noise, but formally assuming it extends to this setting, one would find
\begin{equation}\label{eq:app_proto_balance}
    \frac{\dd}{\dd t} \mathbb{E} [\langle G^s \ast \rho_t, \rho_t \rangle]
    = -2 (d-2s) (s+\alpha-1)\, \mathbb{E}\bigg[ \int_{\mathbb{R}^d \times
  \mathbb{R}^d} | x - y |^{2 s + 2 \alpha - 2 - d} \rho_t (x) \rho_t (y)
  \dd x \dd y \bigg].
\end{equation}
To deduce \eqref{eq:app_exact_balance} from \eqref{eq:app_proto_balance}, we need to distinguish two regimes.

%
%
\textbf{Case 1:  $s+\alpha-1>0$.}
Here we can identify $\mathcal{H}^s$ as a multiple of $G^{s+\alpha-1}$.
Recalling that
\begin{equation}\label{eq:app_riesz_relation}
    \| \varphi\tilde\|_{\dot H^s}^2 := \langle G^s\ast \varphi, \varphi\rangle
    = \pi^{d/2}\, 2^{2s}\, \frac{\Gamma(s)}{\Gamma(d/2-s)} \| \varphi\|_{\dot H^s}^2,
\end{equation}
one finds the balance
\begin{equation}\label{eq:app_balance_case1}
    \frac{\dd}{\dd t} \mathbb{E} \big[\| \rho_t\tilde\|_{\dot H^s}^2\big]
    + 2(d-2s)(s+\alpha-1) \mathbb{E} \big[\| \rho_t \tilde\|_{\dot{H}^{1-\alpha-s}}^2\big] = 0.
\end{equation}
The constant $K_{\alpha,s,d}$ can then be recovered by using the relations \eqref{eq:app_riesz_relation}. Notice that for $\alpha=1$, this computation recovers the result from \cite{zelati2023statistically}, which implies exponential decay of negative Sobolev norms (in expectation).

%
%
\textbf{Case 2:  $s+\alpha-1<0$.} Here $\mathcal{H}^s$ is not a Riesz kernel anymore and has a non-integrable singularity in $0$; passing directly to \eqref{eq:app_proto_balance} is too formal and one needs to be more careful, arguing by approximations.
Consider a regularization $G^{s,\varepsilon}$ of $G^s$ (e.g. by cutting off the kernel close to the origin), such that $G^{s,\varepsilon}$ is a smooth, radial function.
Then $D^2 G^{s,\varepsilon}$ is also even and so must be
\begin{align*}
    \mathcal{H}^{s,\varepsilon} := ((Q (0) - Q (z)) : D^2 G^{s,\varepsilon} (z) .
\end{align*}
%
%
Since $G^s(z) = |z|^{2s-d}$, so that
$| \cH^s (z) | \lesssim |z|^{2s+2\alpha-2-d}$;
the approximations $G^{s,\varepsilon}$ can be constructed in such a way that we still have the uniform-in-$\varepsilon$ pointwise bounds
\begin{align*}
    G^{s,\eps}(z)\lesssim  |z|^{2s-d}, \quad
    |\nabla G^{s,\eps}(z)|\lesssim  |z|^{2s-d-1}, \quad
    |\cH^{s,\varepsilon}(z)| \lesssim |z|^{2s+2\alpha-2-d}.
\end{align*}
Notice moreover that, since $Q$ is divergence-free and $s+\alpha<1$, the kernel $\cH^{s,\varepsilon}$ is mean-free: by the divergence theorem
\begin{align*}
    \left| \int_{\mathbb{R}^d} \cH^{s,\eps}(z) \dd z \right|
    & = \lim_{R \rightarrow \infty} \left| \int_{B_R} \nabla \cdot ((Q (0) - Q (z)) \nabla G^{s,\varepsilon} (z)) \dd z \right|\\
    & \leq  \lim_{R \rightarrow \infty} \int_{\partial B_R} | ((Q (0) - Q (z)) \nabla G^{s,\varepsilon} (z))| \dd z
    \lesssim \lim_{R \rightarrow \infty} R^{d-1} R^{2\alpha} R^{2s-d-1} = 0.
\end{align*}
Combining the even symmetry and mean-free properties of $\cH^{s,\eps}$, we get the following cancellations:
\begin{equation}\label{eq:app_cancellation}\begin{split}
  \langle \cH^{s,\varepsilon} \ast \rho_t, \rho_t \rangle & = \int_{\mathbb{R}^{d
  \times d}} H^{\varepsilon} (x - y) \rho_t (x) \rho_t (y) \dd x \dd y\\
  & = \int_{\mathbb{R}^{d \times d}} \cH^{s,\varepsilon} (x - y) (\rho_t (x) -
  \rho_t (y)) \rho_t (y) \dd x \dd y\\
  & = - \frac{1}{2} \int_{\mathbb{R}^{d \times d}} \cH^{s,\varepsilon} (x - y) |
  \rho_t (x) - \rho_t (y) |^2 \dd x \dd y.
\end{split}\end{equation}
If $\rho$ is regular enough (e.g. $\rho \in L^2_t H^{1-\alpha-}$, which a posteriori is indeed true, or arguing by approximation with the viscous SPDE), we can then pass to the limit as $\varepsilon \rightarrow 0^+$, combining formulas \eqref{eq:app_kernel_balance}, \eqref{eq:formula.H} and \eqref{eq:app_cancellation}, to find
\begin{align*}
  \frac{\dd}{\dd t} \mathbb{E} [\| \rho_t \tilde\|_{\dot{H}^{- s}}^2]
  = - (d-2s) (1-\alpha-s) \mathbb{E} \left[ \int_{\mathbb{R}^{d \times d}}   \frac{| \rho_t (x) - \rho_t (y) |^2}{| x - y |^{d + 2 (1 - \alpha - s)}} \dd x \dd y \right].
\end{align*}
The right-hand side is now a Gagliardo--Nirenberg seminorm which is equivalent to $\| \rho_t\|_{\dot H^{1-\alpha-s}}$, providing the desired conclusion \eqref{eq:app_kernel_balance} also in this case.

\begin{acknowledgements}
    LG was supported by the SNSF Grant 182565 and by the Swiss State Secretariat for Education, Research and lnnovation (SERI) under contract number MB22.00034 through the project TENSE. LG and MM acknowledge support from the Italian Ministry of Research through the project PRIN 2022 ``Noise in fluid dynamics and related models'', project number I53D23002270006. MM acknowledges support from Istituto Nazionale di Alta Matematica (INdAM) 
    through the project GNAMPA 2024 ``Fluidodinamica stocastica'', project number E53C23001670001. LG, FG and MM are members of Istituto Nazionale di Alta Matematica (INdAM), group GNAMPA.

    The authors are grateful to Theodore Drivas and Umberto Pappalettera for several fruitful and stimulating discussions on the topic.
    We thank Marco Bagnara for raising an interesting point concerning Remark \ref{rem:optimality}-ii).
\end{acknowledgements}

\end{document}